\documentclass[a4paper, 11pt, reqno]{amsart}

\usepackage{amsmath,amsthm,amssymb,url,framed,bbm,enumitem}

\usepackage[utf8]{inputenc}
\usepackage{psfrag}
\usepackage{ifpdf}
\ifpdf\usepackage{pst-pdf}\else\fi

\usepackage[style=numeric, minnames=4, maxnames=4,
            doi=false, url=false, isbn=false,
            giveninits=true,
            sortcites=true,
            backend=bibtex]{biblatex}      

\usepackage{graphicx}                     
\graphicspath{{figures/}}

\usepackage[all]{xy}                      
\usepackage{mathrsfs}                     
\usepackage{aliascnt}  
\usepackage[breaklinks=true]{hyperref}          

\newtheoremstyle{convenientthm}%
  {3pt}
  {3pt}
  {\itshape}
  {}
  {\bfseries}
  {.}
  {.5em}
  {\thmnumber{#2 }\thmname{#1}\thmnote{. #3}}

\newtheoremstyle{convenientplain}%
  {3pt}
  {3pt}
  {}
  {}
  {\bfseries}
  {.}
  {.5em}
  {\thmnumber{#2 }\thmname{#1}\thmnote{. #3}}

\theoremstyle{convenientthm}

\newaliascnt{theorem}{subsection}
\newtheorem{theorem}[theorem]{Theorem}
\aliascntresetthe{theorem}

\newtheorem*{theorem*}{Theorem}
\newtheorem*{main*}{Main Theorem}

\newaliascnt{result}{subsection}
\newtheorem{result}[result]{Result}
\aliascntresetthe{result}

\newaliascnt{lemma}{subsection}
\newtheorem{lemma}[lemma]{Lemma}
\aliascntresetthe{lemma}

\newaliascnt{corollary}{subsection}
\newtheorem{corollary}[corollary]{Corollary}
\aliascntresetthe{corollary}

\theoremstyle{convenientplain}

\newaliascnt{definition}{subsection}
\newtheorem{definition}[definition]{Definition}
\aliascntresetthe{definition}

\newaliascnt{remark}{subsection}
\newtheorem{remark}[remark]{Remark}
\aliascntresetthe{remark}

\newaliascnt{example}{subsection}
\newtheorem{example}[example]{Example}
\aliascntresetthe{example}

\let\oldautoref\autoref
\makeatletter
\renewcommand\autoref[1]{\@first@ref#1,@}
\def\@throw@dot#1.#2@{#1}
\def\@set@refname#1{
    \edef\@tmp{\getrefbykeydefault{#1}{anchor}{}}%
    \def\@refname{\@nameuse{\expandafter\@throw@dot\@tmp.@autorefname}s}%
}
\def\@first@ref#1,#2{%
  \ifx#2@\oldautoref{#1}\let\@secondref\@gobble
  \else%
    \@set@refname{#1}
    \@refname~\ref{#1}
    \let\@secondref\@second@ref
  \fi%
  \@secondref#2%
}
\def\@second@ref#1,#2{%
  \ifx#2@ and~\ref{#1}\let\@nextref\@gobble
  \else, \ref{#1}
    \let\@nextref\@next@ref
  \fi%
  \@nextref#2%
}
\def\@next@ref#1,#2{%
   \ifx#2@, and~\ref{#1}\let\@nextref\@gobble
   \else, \ref{#1}
   \fi%
   \@nextref#2%
}
\makeatother

\let\oldtheequation\theequation
\makeatletter
\def\tagform@#1{\maketag@@@{\ignorespaces#1\unskip\@@italiccorr}}
\renewcommand{\theequation}{(\oldtheequation)}
\makeatother

\def\al{\alpha}
\def\be{\beta}

\def\de{\delta}

\def\th{\theta}

\def\ka{\kappa}
\def\la{\lambda}

\def\si{\sigma}

\def\ph{\varphi}

\def\Ga{\Gamma}

\let\on=\operatorname

\let\mc=\mathcalMorawec

\def\inv{^{-1}}
\def\x{\times}
\def\p{\partial}

\def\R{{\mathbb R}}

\def\exp{\operatorname{exp}}

\def\ud{\,\mathrm{d}}
\def\Id{\on{Id}}

\def\ImmL{\on{Imm}}

\def\Diff{\on{Diff}}

\def\Tra{{\on{Tra}}}

\def\Lip{W^{1,\infty}}
\def\ImmL{\mathcal I^{1,\infty}}
\def\Land{\on{Land}}

\let\mc=\mathcal

\DeclareMathOperator*{\esssup}{ess\,sup}
\DeclareMathOperator*{\essinf}{ess\,inf}

\allowdisplaybreaks

\author[M. Bauer]{Martin Bauer}
\address{Faculty for Mathematics, Florida State University,  USA}
\email{bauer@math.fsu.edu}

\author[M. Bruveris]{Martins Bruveris}
\address{Faculty for Mathematics, Brunel University,  UK}
\email{martins.bruveris@brunel.ac.uk }

\author[P. Harms]{Philipp Harms}
\address{Faculty for Mathematics, Freiburg University, Germany}
\email{philipp.harms@stochastik.uni-freiburg.de}

\author[P. Michor]{Peter W. Michor}
\address{Department of Mathematics, University of Vienna, Austria}
\email{peter.michor@univie.ac.at}

\begin{document}

\title{Soliton solutions for the elastic metric on spaces of curves}

\begin{abstract}
In this article we investigate a first order reparametrization-invariant Sobolev metric on the space of immersed curves.
Motivated by applications in shape analysis where discretizations of this infinite-dimensional space are needed, we extend this metric to the space of Lipschitz curves, establish the wellposedness of the geodesic equation thereon, and show that the space of piecewise linear curves is a totally geodesic submanifold. 
Thus, piecewise linear curves are natural finite elements for the discretization of the geodesic equation. 
Interestingly, geodesics in this space can be seen as soliton solutions of the geodesic equation, which were not known to exist for reparametrization-invariant Sobolev metrics on spaces of curves.
\end{abstract}

\maketitle
\setcounter{tocdepth}{1}
\tableofcontents
\section{Introduction}

Geometric shapes can be studied mathematically by viewing them as elements of a Riemannian manifold, which is typically infinite-dimensional. The geodesic distance between shapes is then used as a measure of their dissimilarity. For numerical purposes, shapes need to have a representation in a finite-dimensional space, and a particularly favorable situation arises if this space is a totally geodesic submanifold. In this case geodesics, geodesic distances, and Riemannian curvature in the submanifold coincide (locally) with the corresponding objects in the infinite-dimensional space; there is no discretization error. In this work we show the following result. 

\begin{main*}
The reparametrization-invariant $H^1$-metric 
\begin{equation}\label{equ:i:metric}
G_c(h,k) = \frac{1}{\on{\ell_c}}\int_{S^1} \langle D_s h, D_s k \rangle \ud s
= \frac{1}{\on{\ell_c}}\int_{S^1} \frac{1}{|c_\theta|} \langle h_\theta, k_\theta \rangle \ud \theta\,,
\end{equation}
on the space of immersed closed Lipschitz curves modulo translations possesses finite-dimensional totally geodesic submanifolds, which correspond to finite-element discretizations. The geodesics on these submanifolds turn out to be solitons in the sense that their momenta are sums of delta distributions, which are carried along with the flow.
\end{main*}

The result is established in \autoref{thm:s:totally} and \autoref{cor:s:soliton} below. The notation is explained in \autoref{sec:m}, and the metric is defined rigorously in \autoref{def:m:metric}. 
An introduction to shape analysis and further references for Sobolev metrics can be found in \cite{bauer2014overview}. 

\subsection*{Totally geodesic submanifolds} 

The existence of totally geodesic submanifolds is surprising; it seems to be the exception rather than the rule, at least in the context of shape spaces of immersions and reparametrization-invariant Sobolev metrics. We now explain this in more details.

We are not aware of any reparametrization-invariant metric of order other than one which admits non-trivial totally geodesic subspaces, cf.\@ \autoref{rem:totallygeodesic}.  
We believe, however, that the result does extend to some first-order metrics closely related to \eqref{equ:i:metric}. Examples in this direction are the non scale-invariant $H^1$-metric
\begin{equation*}
G_c(h,k) = \int_{S^1} \langle D_s h, D_s k \rangle \ud s
\end{equation*}
and the elastic metric on planar curves 
\begin{equation*}
G_c(h,k) =  \int_{S^1} \Big(a^2\langle D_s h, v\rangle \langle D_s k,v \rangle + b^2 \langle D_s h, n\rangle \langle D_s k,n \rangle\Big) \ud s\,.
\end{equation*}
In the last equation, $a,b \in \mathbb R$ are constants and $v,n$ are the velocity and normal vector fields to the planar curve $c$. Many of these metrics have in common that there exist isometries to well-known spaces such as spheres, Stiefel manifolds, or submanifolds thereof  \cite{younes2008metric,Jermyn2011,Bauer2014b,lahiri2015precise}, where the existence of totally geodesic subspaces can be studied from an alternative perspective.

\subsection*{Reparametrization-invariance} 

The result is trivial for the flat $L^2$-metric, which does not use the arc-length measure, and variations of it; these are not invariant with respect to reparametrizations. We are, however, not interested in these metrics because they do not induce meaningful (or even well-defined) metrics on the quotient space of immersions modulo reparametrizations. This quotient space is the natural setting for applications in shape analysis, and reparametrization-invariant Sobolev metrics thereon have been used successfully in many applications \cite{LKS2014,Klassen2004,sundaramoorthi2011new,YezziMennucci2005,Esl2014,bauer2017numerical}.

\subsection*{Solitons} 

Soliton solutions were investigated in various contexts. In the context of wave equations, solitons are isolated waves which maintain their shape while traveling at constant speed \cite{zabusky1965interaction, dodd1982solitons}. 

An alternative notion of solitons arises in geometric mechanics, where solutions of a Hamiltonian system are called solitons if their momenta are sums of delta distributions \cite{mumford2013euler}. This is the notion we use in this work; we refer to \cite{michor2006some,michor2007overview} for a Hamiltonian description of shape analysis. Solitons in the sense of geometric mechanics were found for metrics induced by reproducing kernels on diffeomorphism groups \cite{Joshi2000,Younes2010,Micheli2012}, but not yet on spaces of immersions as in this work. We describe a connection of our approach to soliton solutions on diffeomorphism groups in \autoref{sec:l}.

\subsection*{Structure of the article}

The paper is structured as follows. In \autoref{sec:m} we introduce a first-order Sobolev metric on the space of Lipschitz curves and prove that the geodesic equation is well-posed using a geometric method which goes back to Ebin and Marsden \cite{Ebin1970}. In \autoref{sec:p} we study the subspace of piecewise linear curves and equip it with the induced metric of \autoref{sec:m}. \autoref{sec:s} contains our main results: we show that the manifold of piecewise linear curves is totally geodesic and illustrate the soliton-like behavior of geodesics. \autoref{sec:h,sec:l} give a Hamiltonian perspective and establish some relations to  LDDMM metrics on landmark spaces. 

\section{A first order metric on Lipschitz curves}\label{sec:m}

In this section we define a reparametrization-invariant smooth weak Riemannian metric on the space of closed Lipschitz curves modulo translations and establish the well-posedness of the geodesic equation. 

\begin{definition}\label{def:m:lipschitz}
Let $S^1=\R/(2\pi\mathbb Z)$ be the unit circle and $d \in \mathbb N\backslash\{0,1\}$. Let $\Lip$ be the Banach space of Lipschitz continuous functions $c\colon S^1 \to \R^d$, endowed with the norm 
$\|c\|_{\Lip} = \|c\|_{L^\infty}+\|c_\theta\|_{L^\infty}$, where the subscript $\theta$ denotes a derivative. The Banach space $\Lip$ contains the space of Lipschitz continuous immersions 
\begin{equation*}
\ImmL:=\left\{c\in \Lip:\;   \on{essinf}_\theta |c_\theta|>0 \right\}\,.
\end{equation*}
Let $\Tra \cong \R^d$ be the translation group acting on $\Lip$ and $\ImmL$. We will always identify the corresponding quotient spaces as follows:
\begin{align*}
\ImmL/\Tra \cong \ImmL_0 &:=\left\{c\in \ImmL:\;  \int_{S^1}c(\theta)\ud \theta=0 \right\}\,,
\\
\Lip /\Tra \cong\Lip_0 &:= \left\{c\in \Lip:\;  \int_{S^1}c(\theta)\ud \theta=0 \right\}\,.
\end{align*} 
Moreover, we make the convention that all function spaces consist of functions from $S^1$ to $\R^d$, unless another domain or range is specified explicitly. 
\end{definition}

\begin{theorem}\label{thm:m:manifold}
The spaces $\ImmL$ and $\ImmL/\Tra$ are open subsets of the Banach spaces $\Lip$ and $\Lip/\Tra$ and therefore Banach manifolds with tangent bundles $\ImmL \x \Lip$ and $\ImmL_0 \x \Lip_0$, respectively.
\end{theorem}

\begin{proof}
The expression $\essinf_\theta |c_\theta|$ is continuous in $c \in  \Lip$. To see this let $c,\tilde c \in \Lip$ and $\theta\in S^1$. Then
\begin{align*}
|\tilde c_\theta(\theta)| \geq |c_\theta(\theta)|-|\tilde c_\theta(\theta)-c_\theta(\theta)| \geq |c_\theta(\theta)|-\|\tilde c-c\|_{\Lip}
\end{align*}
and consequently
\begin{align*}
\essinf_\theta |\tilde c_\theta(\theta)| \geq \essinf_\theta |c_\theta(\theta)|-\|\tilde c-c\|_{\Lip}.
\end{align*}
Interchanging the roles of $c$ and $\tilde c$ leads to 
\begin{equation*}
\left|\essinf_\theta |\tilde c_\theta(\theta)|-\essinf_\theta |c_\theta(\theta)|\right| \leq \|\tilde c-c\|_{\Lip}.
\end{equation*}
This proves that the mapping $c\mapsto \essinf |c_\theta|$ is Lipschitz on $\Lip$. Thus, $\ImmL$ is an open subset of $\Lip$, and therefore a Banach manifold. The quotient $\Lip/\Tra$ is Banach because $\Tra$ is a closed subspace of the Banach space $\Lip$. As a topological space it is isomorphic to $\Lip_0$. Similarly, $\ImmL/\Tra$ can be identified with $\ImmL_0$, which is an open subset of $\Lip_0$. 
\end{proof}

\begin{remark}
Besides $\ImmL_0$, several other spaces could be used as alternative representations of the quotient space $\ImmL/\Tra$. For example, one could consider all immersions that fix some point $\theta_0$, or all immersions whose center of mass is zero, yielding the spaces
\begin{equation*}
\ImmL_{1} :=\left\{c\in \ImmL:\;  c(\theta_0)=0 \right\}\,,\;
\ImmL_{2} :=\left\{c\in \ImmL:\;  \int c(\theta) |c_\theta| \ud \theta=0 \right\}\,.
\end{equation*}
The particular choice of $\ImmL_0$ is useful in the Hamiltonian description in \autoref{sec:h}. Another possibility is to consider the image $\mc L^\infty_0$ of either of these spaces under the mapping $c\mapsto c_\theta$, i.e., 
\begin{equation*}
\mc L^\infty_0 := \left\{q \in L^\infty: \essinf_{\theta\in S^1} |q(\theta)|>0,\;\int q(\th) \ud \theta=0\right\}\,. 
\end{equation*}
Note, that the second condition ensures that each element of $\mc L^\infty_0$ corresponds to a closed curve.
\end{remark}

\begin{definition}
\label{def:m:metric}
For each $c\in \ImmL$ and $h,k\in \Lip$ we define the bilinear form
\begin{equation*}
G_c(h,k) = \frac{1}{\on{\ell_c}}\int_{S^1} \langle D_s h, D_s k \rangle \ud s
= \frac{1}{\on{\ell_c}}\int_{S^1} \frac{1}{|c_\theta|} \langle h_\theta, k_\theta \rangle \ud \theta\,,
\end{equation*}
where $\ud s=|c_\theta|\ud\theta$ and  $D_s=\frac{1}{|c_\theta|}\partial_{\theta}$ denote differentiation and integration with respect to arc length and $\ell_c=\int_{S^1}\ud s$ is the length of $c$. 
\end{definition}

Note that the bilinear form $G_c$ is degenerate because $G_c(h,h)=0$ for each constant $h:S^1\to\R^d$. It is, however, non-degenerate if translations are factored out, as the following theorem shows.

\begin{lemma}\label{lem:m:smooth}
$G$ is a smooth weak Riemannian metric on $\ImmL_0$.
\end{lemma}

\begin{proof}
If $G_c(h,h)=0$ for some $c \in \ImmL_0$ and $h \in \Lip_0$, then $h_\theta=0$ almost everywhere. It follows that $h=0$ because $\int_{S^1} h(\theta)\ud \theta=0$ by assumption. Therefore, $G$ is non-degenerate. The smoothness of $G$ is a consequence of \autoref{cor:a:smooth}.
\end{proof}

\begin{remark}
Note that the metric $G$ is invariant under the action of the diffeomorphism group $\Diff(S^1)$ on $\ImmL$:
\begin{align*}
G_{c\circ\varphi}(h\circ\varphi,k\circ\varphi) &=  \frac{1}{\on{\ell_{c\circ\varphi}}}\int_{S^1} \frac{1}{|c_\theta\circ\varphi.\varphi_\theta|} \langle h_\theta\circ\varphi.\varphi_\theta, k_\theta\circ\varphi.\varphi_\theta \rangle \ud \theta\\
&=  \frac{1}{\on{\ell_{c}}}\int_{S^1} \frac{1}{|c_\theta|} \langle h_\theta, k_\theta \rangle \ud \theta=G_c(h,k)\,.
\end{align*}
Moreover, note that $G$ is invariant under scalings $x\mapsto \lambda x$, $\lambda >0$, $x\in\R^d$.
\end{remark}

To formulate the geodesic equation, which is our next goal, we need to invert the operator $D_s$ on a suitably restricted domain. This is achieved by the following lemma.

\begin{lemma}\label{lem:m:Ds}
For each $c \in \ImmL_0$ the following diagram is commutative,
\begin{equation*}
\xymatrix{
W^{1,\infty} \ar@{->>}[r]^-{\pi_1} \ar[d]^{D_s}
& W^{1,\infty}_0 := \left\{h \in W^{1,\infty}: \int h\ud \theta =0\right\} \ar@{^{(}->}[r]^-{\iota_1} \ar@<0.5ex>[d]^{D_s} 
& W^{1,\infty} \ar[d]^{D_s}
\\
W^{0,\infty} \ar@{->>}[r]^-{\pi_0}
& W^{0,\infty}_0 := \left\{h \in W^{0,\infty}: \int h\ud s =0\right\} \ar@{^{(}->}[r]^-{\iota_0} \ar@<0.5ex>[u]^{D_s^{-1}}
& W^{0,\infty},
}
\end{equation*}	
where $\pi_0$ is the $L^2(\!\ud s)$-orthogonal projection, $\pi_1$ is the $L^2(\!\ud \theta)$-orthogonal projection, $\iota_0$ and $\iota_1$ are inclusions. Note: the space $W^{0,\infty}_0$ depends on $c$.
\end{lemma}

\begin{proof}
The commutativity of the diagram, including the existence of $D_s^{-1}$, can be verified using the explicit formulas
\begin{align*}
(D_s h)(\theta) &= h_\theta(\theta) |c_\theta(\theta)|^{-1}\,, 
\\
(D_s\inv k)(\theta) &= \int_0^\theta k(\eta) |c_\theta(\eta)| \ud \eta - \frac{1}{2\pi}\int_{S^1} \int_0^\zeta k(\eta) |c_\theta(\eta)| \ud \eta \ud \zeta\,,
\\
(\pi_1 h)(\theta) &= h(\theta) - \frac{1}{2\pi}\int_{S^1} h(\eta) \ud \eta\,,
\\
(\pi_0 h)(\theta) &= h(\theta) - \frac{1}{2\pi}\int_{S^1} h(\eta) |c_\theta(\eta)| \ud \eta\,,
\end{align*}
where $\theta \in [0,2\pi)$. 
\end{proof}

We recall that geodesics are critical points of the energy functional. Under general weak Riemannian metrics the geodesic equation might not exist, i.e., it might not be possible to express the first-order condition for critical points as a differential equation of second order in time. This is not the case for the metric $G$, as we will show now. Our proof avoids second derivatives and therefore allows us to work on the space of Lipschitz immersions. The theorem is consistent with the geodesic equation derived in \cite[App. I]{younes2008metric} for smooth immersions. This can be seen from the relation $D_s^{-2} (\ka n) = c$. 

\begin{theorem}\label{thm:m:geo}
The geodesic equation of the weak Riemannian metric $G$ on $\ImmL_0$ exists and is given by
\begin{multline}\label{equ:m:geo}
c_{tt}=
G_c(c,c_t) c_t -\frac12 G_c(c_t,c_t)c 
\\
+ D_s^{-1}\pi_0\! \left(\langle D_s c,D_s c_{t}\rangle D_s c_{t}-\frac12 | D_s c_t|^2 D_s c\right),
\end{multline}
where $D_s^{-1}$ and $\pi_0$ are defined in \autoref{lem:m:Ds}.
\end{theorem}

\begin{proof}
The Riemannian energy of a path $c = c(t, \th)$ is
\begin{equation*}
E(c)=
\frac12  \int_{0}^1 G_c(c_t,c_t)\ud t 
=
\frac12\int_{0}^1\frac{1}{\ell_c} \int_{S^1} \langle c_{t\theta}, c_{t\theta} \rangle 
\frac{1}{|c_\theta|}\ud\theta \ud t\,.
\end{equation*}
Varying $c$ in the direction $h = h(t,\th)$ with $h(0)=h(1)=0$ yields
\begin{align*}
\ud E(c).h
=&\frac12\int_0^1\bigg(
-\frac{1}{\ell_c^2}\int_{S^1}\frac{\langle c_{\theta}, h_{\theta} \rangle}{|c_{\theta}|}\ud\theta
\int_{S^1} \frac{\langle c_{t\theta}, c_{t\theta} \rangle}{|c_{\theta}|}\ud\theta
\\&\qquad
 +\frac{2}{\ell_c} \int_{S^1} 
\frac{\langle c_{t\theta}, h_{t\theta} \rangle}{|c_\theta|}\ud\theta
-\frac1{\ell_c} \int_{S^1} \langle c_{t\theta}, c_{t\theta} \rangle 
\frac{\langle c_\theta,h_\theta\rangle}{|c_\theta|^3}\ud\theta
\bigg) \ud t
\,.
\end{align*}
In the second integral, integration by parts with respect to $t$ can be used to eliminate the 
time-derivative of $h$:
\begin{multline*}
\int_0^1\frac{2}{\ell_c} \int_{S^1} 
\frac{\langle c_{t\theta}, h_{t\theta} \rangle}{|c_\theta|}\ud\theta \ud t
=
\int_0^1 \bigg( \frac{2}{\ell_c^2}\int_{S^1}
\frac{\langle c_{\theta}, c_{t\theta} \rangle}{|c_\theta|} \ud\theta
\int_{S^1} \frac{\langle c_{t\theta}, h_{\theta} \rangle}{|c_\theta|}\ud\theta
\\
-\frac{2}{\ell_c} \int_{S^1}  
\frac{\langle c_{tt\theta}, h_{\theta} \rangle}{|c_\theta|}\ud\theta
+\frac{2}{\ell_c} \int_{S^1} \langle c_{t\theta}, h_{\theta} \rangle 
\frac{\langle c_\theta,c_{t\theta}\rangle}{|c_\theta|^3}\ud\theta \bigg) \ud t
\end{multline*}
Note that the boundary terms vanish because $h(1)=h(0)=0$. Thus, 
\begin{align*}
\ud E(c).h
=&\frac12\int_0^1\bigg(
-\frac{1}{\ell_c^2}\int_{S^1}\frac{\langle c_{\theta}, h_{\theta} \rangle}{|c_{\theta}|}\ud\theta
\int_{S^1} \frac{\langle c_{t\theta}, c_{t\theta} \rangle}{|c_{\theta}|}\ud\theta
\\&
{}+\frac{2}{\ell_c^2}\int_{S^1}
\frac{\langle c_{\theta}, c_{t\theta} \rangle}{|c_\theta|} \ud\theta
\int_{S^1} \frac{\langle c_{t\theta}, h_{\theta} \rangle}{|c_\theta|}\ud\theta
-\frac{2}{\ell_c} \int_{S^1}  
\frac{\langle c_{tt\theta}, h_{\theta} \rangle}{|c_\theta|}\ud\theta
\\&
+\frac{2}{\ell_c} \int_{S^1} \langle c_{t\theta}, h_{\theta} \rangle 
\frac{\langle c_\theta,c_{t\theta}\rangle}{|c_\theta|^3}\ud\theta 
-\frac1{\ell_c} \int_{S^1} \langle c_{t\theta}, c_{t\theta} \rangle 
\frac{\langle c_\theta,h_\theta\rangle}{|c_\theta|^3}\ud\theta
\bigg) \ud t
\,.
\end{align*}
In terms of $D_s$ and $\mathrm{d}s$ this reads as
\begin{align*}
\ud E(c).h
=&\frac12\int_0^1\bigg(
-\frac{1}{\ell_c^2}\int_{S^1}\langle D_s c,D_s h \rangle \ud s
\int_{S^1} \langle D_s c_{t}, D_s c_{t} \rangle\ud s
\\&
+\frac{2}{\ell_c^2}\int_{S^1}\langle D_s c, D_s c_{t} \rangle \ud s
\int_{S^1} \langle D_s c_{t},D_s h \rangle \ud s
-\frac{2}{\ell_c} \int_{S^1} \langle D_s c_{tt}, D_s h \rangle \ud s
\\&
+\frac{2}{\ell_c} \int_{S^1} \langle D_s c_{t}, D_s h \rangle 
\langle D_s c,D_s c_{t}\rangle\ud s
\\&
-\frac1{\ell_c} \int_{S^1} \langle D_s c_{t}, D_s c_{t} \rangle 
\langle D_s c,D_s h\rangle\ud s
\bigg) \!\ud t
\,.
\end{align*}
For the last two summands we will use the following relation, which follows from the definition of the metric $G$ and of the mappings $D_s\inv$ and $\pi_0$ of \autoref{lem:m:Ds}: it holds for all $k \in \Lip$ that
\begin{align*}
\frac{1}{\ell_c} \int_{S^1} \langle k,D_s h\rangle \ud s
&=
\frac{1}{\ell_c} \int_{S^1} \langle \pi_0 k,D_s h\rangle \ud s
\\&=
\frac{1}{\ell_c} \int_{S^1} \langle D_sD_s^{-1}\pi_0 k,D_s h\rangle \ud s
=
G_c(D_s^{-1}\pi_0k,h)\,.
\end{align*}
This allows us to rewrite $\ud E(c).h$ as
\begin{align*}
\ud E(&c).h
=\frac 12 \int_0^1\bigg( 
-G_c(c,h) G_c(c_t,c_t)+2G_c(c,c_t) G_c(c_t,h)-2 G_c(c_{tt},h)
\\&
+G_c\Big( D_s^{-1}\pi_0 \big(2\langle D_s c,D_s c_{t}\rangle D_s c_{t}-\langle D_s c_{t},D_s c_{t}\rangle D_s c\big), h \Big)
\!\bigg)\!\ud t\,.
\end{align*}
Therefore $\ud E(c).h = 0$ if and only if~\eqref{equ:m:geo} is satisfied.
\end{proof}

The well-posedness of the geodesic equation in the smooth category and on Sobolev immersions of order $k > 5/2$ has been shown in \cite{younes2008metric}. Here we extend this result to Lipschitz immersions. Our proof also carries over to the space of Sobolev immersions of order $k>\frac32$. 

\begin{theorem}\label{thm:m:exp}
The initial value problem for the geodesic equation \eqref{equ:m:geo}
has unique local solutions in the Banach manifold $\mc I^{1,\infty}_0$. The solutions depend smoothly on $t$ and on the initial
conditions $c(0,\cdot)$ and $c_t(0,\cdot)$.
Moreover, the Riemannian exponential mapping $\exp$ exists 
and is smooth on a neighborhood of the zero section in the tangent bundle, and the map $(c, h) \mapsto (c, \exp_c(h))$ is a local diffeomorphism from a (possibly smaller) neighborhood of the zero section to a neighborhood of the diagonal in the product $\mc I^{1,\infty}_0 \times \mc I^{1,\infty}_0$.
\end{theorem} 

\begin{proof}
We interpret the geodesic equation as an ODE on the Banach manifold $T\ImmL_0$,
\begin{equation*}\left\{\begin{aligned}
c_t &= u\,, \\
u_t &= \Ga_c(u,u)\,,
\end{aligned}\right.\end{equation*}
where the Christoffel symbol $\Ga_c(h,h)$ is given by
\[
\Ga_c(u,u) = G_c(c,u) u -\frac12 G_c(u,u)c +
 D_s^{-1}\pi_0\! \left(\langle D_s c,D_s u\rangle D_s u-\frac12 | D_s u|^2 D_s c\right).
\]
The map $(c,u) \mapsto \Gamma_c(u,u)$ is smooth because all spaces and mappings in the diagram in \autoref{lem:m:Ds} depend smoothly on $c$ in the following sense: the spaces in the diagram are fibers of smooth vector bundles over $\ImmL_0$, and the mappings in the diagram are smooth bundle homomorphisms. This follows from \autoref{cor:a:smooth} using the global vector bundle chart $(c,h)\mapsto (c,h |c_\theta|)$ for $W^{0,\infty}_0$. Hence we obtain short time existence of solutions of the geodesic equation by the theorem of Picard-Lindel\"of. Furthermore, the solutions depend smoothly on the initial values. The local invertibility of the exponential map follows by standard arguments and the implicit function theorem.
\end{proof}

\section{The submanifold of piecewise linear curves}\label{sec:p}

\begin{definition}\label{def:p:piecewise}
Let $n \in \mathbb N_{>0}$ and $0 = \theta^1<\ldots< \theta^{n+1} = 2\pi$ be fixed such that $|\theta^{i+1}-\theta^i|=2\pi/n$ for all $i \in \{1,\dots,n\}$. Then $\theta^1$ and $\theta^{n+1}$ are equal as elements of $S^1=\R/(2\pi\mathbb Z)$. We write $[\theta^i,\theta^{i+1}]$ for the interval in both $\R$ and $S^1$, and we use the word ``piecewise'' to mean piecewise with respect to the grid $\theta^i$. We let $\mathcal P^0$ denote the set of piecewise constant left-continuous functions in $W^{0,\infty}$, $\mathcal P^1$ the set of piecewise linear functions in $W^{1,\infty}$, and $\mathcal P\mathcal I^1$ the set of piecewise linear immersions in $\ImmL$. We use subscripts $0$ to denote intersections with $W^{0,\infty}_0$, $W^{1,\infty}_0$, and $\ImmL_0$, respectively. 
For each curve $c \in \mathcal P^1$ we set
\begin{align*}
\ell^i &= |c(\theta^{i+1})-c(\theta^i)|\,, &
\la^i &= \sum_{j=i}^n \ell^j\,.
\end{align*}
\end{definition}

We now present a discrete counterpart of \autoref{lem:m:Ds},  describing the operators $D_s$ and $D_s^{-1}$ on the discretized spaces of curves.

\begin{lemma}\label{lem:p:Ds}
For each $c \in \mathcal P\mathcal I^1$ the following diagram is commutative,
\begin{equation*}
\xymatrix{
\mathcal P^1 \ar@{->>}[r]^-{\pi_1} \ar[d]^{D_s}
& \mathcal P^1_0 := \left\{h \in \mathcal P^1: \int h\ud \theta =0\right\} \ar@{^{(}->}[r]^-{\iota_1} \ar@<0.5ex>[d]^{D_s} 
& \mathcal P^1 \ar[d]^{D_s}
\\
\mathcal P^0 \ar@{->>}[r]^-{\pi_0}
& \mathcal P^0_0 := \left\{k \in \mathcal P^0: \int k\ud s =0\right\} \ar@{^{(}->}[r]^-{\iota_0} \ar@<0.5ex>[u]^{D_s^{-1}}
& \mathcal P^0,
}
\end{equation*}	
where $\pi_0$ is the $L^2(\!\ud s)$-orthogonal projection, $\pi_1$ is the $L^2(\!\ud \theta)$-orthogonal projection, $\iota_0$ and $\iota_1$ are inclusions. Note that the space $\mathcal P^0_0$ depends on $c$. 
\end{lemma}

\begin{proof}
It is straight-forward to verify that the operators in \autoref{lem:m:Ds} restrict to the spaces above.
\end{proof}

We will use the following natural identifications with Euclidean spaces.

\begin{definition}\label{def:p:coordinates}
The spaces $\mathcal P^1$ and $\mathcal P^0$ are naturally isomorphic to $\R^{n\times d}$ via the identification of $h \in \mathcal P^1$ and $k \in \mathcal P^0$ with  
\begin{equation*}
\big(h(\theta^1), \dots, h(\theta^n)\big) \in \R^{n\times d}\,, \qquad 
\big(k(\theta^1), \dots, k(\theta^n)\big) \in \R^{n\times d}\,.
\end{equation*}
By duality we get identifications of $(\mathcal P^1)^*$ and $(\mathcal P^0)^*$ with $\R^{n\times d}$ such that the pairing of dual elements is given by the Euclidean scalar product on $\R^{n\times d}$. Under these identifications the spaces $\mathcal P^1_0$, $\mathcal P^0_0$, $(\mathcal P^1_0)^*$, and $(\mathcal P^0_0)^*$, which can be viewed as subspaces using the inclusion mappings $\iota_1$, $\iota_0$, $\pi_1^*$, and $\pi_0^*$, correspond to the following subspaces of $\R^{n\times d}$:
\begin{align*}
\left\{h \in \mathcal P^1_0 \text{ with vertices } (h^i)_{i=1,\dots,n} \vphantom{\sum_{i=1}^n}\right\} &\stackrel{\cong}{\longrightarrow} \left\{h \in \R^{n\times d}: \sum_{i=1}^n h^i =0 \right\}\,,
\\
\left\{k \in \mathcal P^0_0 \text{ with edges } (k^i)_{i=1,\dots,n}\vphantom{\sum_{i=1}^n}\right\} &\stackrel{\cong}{\longrightarrow} \left\{k \in \R^{n\times d}: \sum_{i=1}^n k^i \ell^i =0 \right\}\,,
\\
\left\{\alpha = \sum_{i=1}^n \alpha^i \delta_{\theta^i} \in (\mathcal P^1_0)^*\right\} &\stackrel{\cong}{\longrightarrow} \left\{\alpha \in \R^{n\times d}: \sum_{i=1}^n \alpha^i =0 \right\}\,,
\\
\left\{\beta = \sum_{i=1}^n \beta^i/\ell^i \mathbbm{1}_{[\theta^i,\theta^{i+1})}\ud s \in (\mathcal P^0_0)^*\right\} &\stackrel{\cong}{\longrightarrow} \left\{\beta \in \R^{n\times d}: \sum_{i=1}^n \beta^i =0 \right\}\,.
\end{align*}
Note that the pairing of dual elements is still given by Euclidean scalar products. (Formally this follows from the relations $\pi_0\circ \iota_0=\Id$, $\pi_1\circ \iota_1=\Id$.)
\end{definition}

The following lemma provides explicit expressions of various operators in the Euclidean coordinates of  \autoref{def:p:coordinates}.

\begin{lemma}\label{lem:p:Ds2}
Under the identifications of \autoref{def:p:coordinates}, the following relations hold for each $h \in \mathcal P^1_0$, $k\in \mathcal P^0_0$, $\beta \in (\mathcal P^0_0)^*$, $\alpha \in (\mathcal P^1_0)^*$, and $i \in \{1,\dots,n\}$:
\begin{align*}
(\pi_1 h)^i &=h^i-\frac1n \sum_{j=1}^n h^j\,,
&
(\pi_0 k)^i &= k^i-\frac{1}{\ell_c}\sum_{j=1}^n k^j\ell^j\,,
\\
(D_s h)^i &= \frac{h^{i+1}-h^i}{\ell^i}\,, 
&
(D_s^{-1}k)^i &= \sum_{j=1}^{i-1} k^j\ell^j - \frac1n\sum_{m=1}^n\sum_{j=1}^{m-1}k^j\ell^j\,,\\
(k\ud s)^i &= k^i\ell^i\,,
&
(\beta/\!\ud s)^i &= \beta^i/\ell^i\,,
\\
(D_s^*\beta)^i &= \beta^{i-1}/\ell^{i-1}-\beta^i/\ell^i 
\, &
((D_s^*)^{-1}\alpha)^i &= \Bigg(\frac{1}{\ell_c}\sum_{j=1}^n \al^j \la^j-\sum_{j=1}^i\al^j\Bigg)\ell^i\,.
\end{align*}
\end{lemma}

\begin{proof}
The formulas for $\pi_1$, $\pi_0$, $D_s$, $\ud s$, and $D_s^*$ follow from \autoref{lem:p:Ds} and \autoref{def:p:coordinates}. The formula for $D_s^{-1}$ can be seen as follows. The relation $k^i=(h^{i+1}-h^i)/\ell^i$ implies that $h^i = \xi+\sum_{j=1}^{i-1}k^j\ell^j$ for some $\xi \in \R^d$. The vector $\xi$ is determined by the condition $\sum_{i=1}^n h^i=0$ and given by $\xi=-\frac1n \sum_{m=1}^n\sum_{j=1}^{m-1}k^j\ell^j$. Similar calculations establish the remaining formulas. 
\end{proof}

The weak Riemannian metric $G$ of \autoref{def:m:metric} can be pulled back to the manifold $\mathcal P\mathcal I^1_0 = \mathcal P^1 \cap \ImmL_0$. This turns $\mathcal P\mathcal I^1_0$ into a Riemannian manifold, which we describe next.

\begin{theorem}
\label{thm:p:metric}
\label{thm:p:mom}
\label{thm:p:cometric}
Under the identifications of \autoref{def:p:coordinates} the metric, momentum mapping, and cometric on $\mathcal P\mathcal I^1_0 \subset \ImmL_0$ are given by 
\begin{align*}
G_c(h,k) &= \frac 1{\ell_c} \sum_{i=1}^n \frac{1}{\ell^i}\langle h^{i+1}-h^i,k^{i+1}-k^i\rangle\,,
\\
\big(\check G_c(h)\big)^i &= \frac{1}{\ell_c}\left(\frac{h^i-h^{i-1}}{\ell^{i-1}}-\frac{h^{i+1}-h^i}{\ell^i}\right)\,,
\\
G^{-1}_c(\alpha,\beta) &= \sum_{i,j=1}^n \big(\lambda^1\lambda^{\max(i,j)}-\lambda^i\lambda^j\big) \langle \alpha^i,\beta^j\rangle\,,
\end{align*}
where $c \in \mathcal P\mathcal I^1_0$, $h,k \in T_c \mathcal P\mathcal I^1_0$, and $\alpha,\beta \in T^*\mathcal P\mathcal I^1_0$.
\end{theorem}

\begin{proof}
By \autoref{lem:p:Ds2} we have
\begin{align*}
(D_sh)^i&=\frac{h^{i+1}-h^i}{\ell^i}\,, 
\qquad\qquad
(D_s h \ud s)^i = h^{i+1}-h^i\,, 
\\
(D_s^*(D_s h\ud s))^i &= \frac{h^i-h^{i-1}}{\ell^{i-1}}-\frac{h^{i+1}-h^i}{\ell^i}\,,
\\
\left((D_s^*)^{-1}\alpha\right)^i &= \Bigg(\frac{1}{\ell_c}\sum_{j=1}^n \al^j \la^j-\sum_{j=1}^i\al^j\Bigg)\ell^i\,,
\\
\left(\frac{(D_s^*)^{-1}\alpha}{\!\ud s}\right)^i 
&= \frac{1}{\ell_c}\sum_{j=1}^n \al^j \la^j-\sum_{j=1}^i\al^j 
= \sum_{j=1}^n \al^j \big(\la^j/\ell_c-\mathbbm{1}_{j\leq i}\big) \,.
\end{align*}
Then the formula for the metric follows from
\begin{equation*}
G_c(h,k)=\frac{1}{\ell_c}\int_{S^1} \left\langle D_s h,D_s k \right\rangle \ud s
=\frac{1}{\ell_c} \sum_{i=1}^{n} \ell^i \left\langle (D_s h)^i, (D_s k)^i\right\rangle,
\end{equation*}
and the formula for the momentum mapping from \autoref{lem:b:mom}. Using \autoref{lem:b:cometric} the cometric is given by
\begin{align*}
G_c^{-1}(\alpha,\beta) &= \ell_c \sum_{k=1}^n \ell^k \left\langle\left(\frac{(D_s^*)^{-1}\alpha}{\!\ud s}\right)^k, \left(\frac{(D_s^*)^{-1}\beta}{\!\ud s}\right)^k \right\rangle
\\&= \frac{1}{\ell_c} \sum_{k=1}^n \ell^k \left\langle\sum_{i=1}^n \al^i \big(\la^i-\ell_c\mathbbm{1}_{i\leq k}\big), \sum_{j=1}^n \be^j \big(\la^j-\ell_c \mathbbm{1}_{j\leq k}\big) \right\rangle
\\&= \sum_{i=1}^n \sum_{j=1}^n \langle\al^i, \be^j\rangle \frac{1}{\ell_c} \sum_{k=1}^n \ell^k \big(\la^i-\ell_c\mathbbm{1}_{i\leq k}\big)  \big(\la^j-\ell_c \mathbbm{1}_{j\leq k}\big) 
\\&=
\sum_{i=1}^n \sum_{j=1}^n \langle\al^i, \be^j\rangle \Bigg(\lambda^i\lambda^j-\lambda^i\lambda^j-\lambda^i\lambda^j+\ell_c \sum_{k=\max(i,j)}^n\ell^k \Bigg)
\\&=
\sum_{i=1}^n \sum_{j=1}^n \langle\al^i, \be^j\rangle \left(\lambda^1\lambda^{\max(i,j)}-\lambda^i\lambda^j \right)\,. \qedhere
\end{align*}
\end{proof}

\section{Soliton solutions of the geodesic equation}\label{sec:s}

In this section we establish our two main results. First, we show that piecewise linear curves are a totally geodesic subspace of the space of Lipschitz curves modulo translations. Second, we prove that the geodesic equation admits soliton solutions. We establish this result by showing that the momentum of a curve is a sum of delta distributions if and only if the velocity is piecewise linear up to a reparametrization. 

\begin{theorem}\label{thm:s:totally}
The space $\mathcal P\mathcal I^1_0$ is a totally geodesic submanifold of dimension $(n-1)\times d$ in the manifold $\ImmL_0$ with weak Riemannian metric $G$, for each $n\in\mathbb N\backslash\{0\}$.
\end{theorem}

\begin{proof}
The space $\mathcal P^1_0$ of piecewise linear functions is a finite-di\-men\-sional linear subspace of $\Lip_0$, hence complemented. Thus, the open subset $\mathcal P\mathcal I^1_0$ is a splitting submanifold of $\ImmL_0$. To show that $\mathcal P\mathcal I^1_0$ is totally geodesic we take a tangent vector $h\in \mathcal P^1_0$ with foot point $c\in \mathcal P\mathcal I^1_0$ and consider the right hand side of the geodesic equation,
\begin{multline*}
\Ga_c(h,h) = G_c(c,h) h -\frac12 G_c(h,h)c 
\\+D_s^{-1}\circ\pi_0 \left(\langle D_s c,D_s h\rangle D_s h-\frac12 | D_s h|^2 D_s c\right).
\end{multline*}
The operator $D_s^{-1} \circ \pi_0: \mathcal P^0 \to \mathcal P^1_0$ maps piecewise constant functions to piecewise linear ones. Moreover, $\mathcal P^0$ is an algebra under pointwise multiplication. Thus, we obtain $\Ga_c(h,h) \in \mathcal P^1_0$. It follows that the geodesic equation restricts to an ODE on the submanifold $T\mathcal P\mathcal I^1_0$, showing that $\mathcal P\mathcal I^1_0$ is totally geodesic.
\end{proof}

\begin{remark}\label{rem:totallygeodesic}
The existence of these totally geodesic submanifolds is highly surprising. We are not aware of any reparametrization-invariant metric of order other than one which admits similar totally geodesic subspaces. This is, however, not to say that there are no other totally geodesic subspaces. For example, every geodesic defines a one-dimensional totally geodesic subspace. Moreover, the set of concentric circles with common center $x\in\R^d$ is a totally geodesic submanifold for many metrics \cite{Salvai2011,Michor2006c,Bauer2012a}. This is the case whenever the rotation group acts isometrically on the space of curves, the reason being that the set of concentric circles is the fixed point set of the rotation group. Under some metrics the set of all circles with arbitrary radius and center is also totally geodesic. These spaces are, however, not useful in numerical applications where one needs discretizations of arbitrary curves.
\end{remark}

\begin{remark}
\autoref{thm:s:totally} can be reformulated for the space of closed curves modulo rotations as follows: 
The metric is invariant under the rotation group and thus it induces a
metric on the quotient space such that the projection is a Riemannian submersion, see \cite{younes2008metric}. 
As rotations leave the space of polygons invariant, our results imply that
polygonal curves are also totally geodesic in the quotient space of curves modulo rotations. 
\end{remark}

\begin{remark}
\autoref{thm:s:totally} can be reformulated for open instead of closed curves as follows: 
if $S^1$ is replaced by $[0,2\pi]$, 
$W^{0,\infty}_0$ is redefined as $W^{0,\infty}$, 
and $\pi_0$ and $\iota_0$ are redefined as identity mappings,  
then \autoref{thm:m:manifold}, \autoref{lem:m:smooth}, \autoref{lem:m:Ds}, \autoref{thm:m:geo}, and \autoref{thm:m:exp}
remain valid, the coordinate expressions of \autoref{sec:p} take a  different form, and \autoref{thm:g:geo} remains valid with $(n-1)\times d$ replaced by $n\times d$. 
\end{remark}

\begin{definition}\label{def:s:soliton}
A soliton is a path $c$ in $\ImmL_0$ whose momentum 
is at all times a sum of delta distributions, i.e., one has for each $t$ that
$$\check G_c(c_t)= \frac{1}{\ell_c}D_s^* (D_s h \ud s) = -\frac{1}{\ell_c}D_s^2 (h \ud s)=\sum_{i=1}^n \langle \al^i(t), \de_{\th^i(t)}\rangle$$
with $\al^i(t) \in \R^d$ and $\th^i(t) \in S^1$. More details on the momentum as an element of $(W^{1,\infty})^*$ can be found in 
\autoref{sec:b}.
\end{definition}

The following lemma 
characterizes all velocities whose momenta are sums of delta distributions.

\begin{lemma}\label{lem:s:mom}
For any $c \in \ImmL_0$ and $h \in \Lip_0$, the momentum 
$\check G_c(h)$ is a sum of delta distributions if and only if $h \circ \ph$ is piecewise linear, where $\ph \in W^{1,\infty}(S^1,S^1)$ is such that $c \circ \ph$ has constant speed. 
\end{lemma}

\begin{proof}
Let $\ph \in W^{1,\infty}(S^1,S^1)$ be such that $c\circ\ph$ has constant speed. We claim that $\check G_c(h)$ is a sum of delta distributions if and only if $\check G_{c \circ \ph}(h\circ \ph)$ is a sum of delta distributions. To see this, assume that $\check G_c(h)=\sum_{i=1}^n\langle\alpha^i,\delta_{\theta^i}\rangle_{\mathbb R^d}$ for some $\alpha^i \in \R^d$, and let $k:S^1\to\mathbb R^d$ be any smooth function. By the reparametrization-invariance of the metric, 
\begin{align*}
\check G_{c\circ\ph}(h\circ \ph)(k) 
&= 
\check G_{c\circ\ph}(h\circ \ph)(k\circ \ph^{-1}\circ \ph)
=
\check G_c(h)(k\circ \ph^{-1})
\\&=
\sum_{i=1}^n \langle \alpha^i,k(\ph^{-1}(\theta^i))\rangle_{\mathbb R^d}
=
\left(\sum_{i=1}^n\langle\alpha^i,\delta_{\ph^{-1}(\theta^i)}\rangle_{\mathbb R^d}\right)(k).
\end{align*}
Therefore $\check G_{c \circ \ph}(h\circ \ph)$ is a sum of delta distributions. Reversing the argument proves the claim. 

It remains to prove the lemma in the case where $c$ has constant speed and $\ph$ is the identity. Then we have $D_s = 2\pi/\ell_c \p_\th$ and therefore $\check G_c(h) = -2\pi/\ell_c^2 h_{\theta\theta}$. It follows that $\check G_c(h)$ is a sum of delta distributions if and only if $h$ is piecewise linear.
\end{proof}

 If $c$ is a piecewise linear curve, then the map $\ph$ mediating between $c$ and the constant speed reparametrization $c \circ \ph$ is also piecewise linear. In this case the second part of \autoref{lem:s:mom} simplifies to: the momentum $\check G_c(h)$ is a sum of delta distributions if and only if $h$ is piecewise linear. Thus, the tangent space to piecewise linear curves corresponds via $\check G_c$ to momenta that are sums of delta distributions. It is therefore natural to search for soliton solutions of the geodesic equation in the submanifold of piecewise linear curves, which is defined next.

\begin{corollary}\label{cor:s:soliton}
Geodesics $c$ in $\mathcal P\mathcal I^1_0$ are soliton solutions of the geodesic equation, i.e., $\check G_c(c_t)$ is a sum of delta distributions. 	
\end{corollary}

\begin{proof}
This follows from \autoref{thm:s:totally} and \autoref{lem:s:mom}.
\end{proof}

\begin{theorem}
\label{thm:g:geo}
In the coordinates of \autoref{def:p:coordinates} the geodesic equation on $\mathcal P\mathcal I^1_0$ is given by the following $n\times d$-dimensional system of ODEs, 
\begin{align*}
\begin{pmatrix}
c^1_{tt}\\\vdots\\c^n_{tt}
\end{pmatrix}
&=
\frac{1}{\ell_c}\sum_{j=1}^{n} \frac{\langle c_t^{j+1}-c_t^{j}, c^{j+1}-c^{j}\rangle}{\ell^j}
\begin{pmatrix}
c^1_{t}\\\vdots\\c^n_{t}
\end{pmatrix}
\\&\qquad\qquad 
-\frac{1}{2\ell_c}\sum_{j=1}^{n} \frac{\langle c_t^{j+1}-c_t^{j}, c_t^{j+1}-c_t^{j}\rangle}{\ell^j} \begin{pmatrix}
c^1\\\vdots\\c^n
\end{pmatrix} 
\\&\qquad\qquad+
D_s^{-1}\circ \pi_0
\begin{pmatrix}
\frac{1}{(\ell^1)^3}\langle c^{2}-c^1,c_t^{2}-c_t^{1}\rangle (c_t^{2}-c_t^{1})\\\vdots\\
\frac{1}{(\ell^n)^3}\langle c^{n+1}-c^n,c_t^{n+1}-c_t^{n}\rangle (c_t^{n+1}-c_t^{n})
\end{pmatrix}
\\&\qquad\qquad-\frac12
D_s^{-1}\circ \pi_0
\begin{pmatrix}
\frac{1}{(\ell^1)^3}\langle c_t^{2}-c_t^1,c_t^{2}-c_t^{1}\rangle (c^{2}-c^{1})\\\vdots\\
\frac{1}{(\ell^n)^3}\langle c_t^{n+1}-c_t^n,c_t^{n+1}-c_t^{n}\rangle (c^{n+1}-c^{n})
\end{pmatrix}
\,,
\end{align*}
where the operators $D_s^{-1}$ and $\pi_0$ are given by \autoref{lem:p:Ds2}.
\end{theorem}

\begin{proof}
As $\mathcal P\mathcal I^1_0$ is totally geodesic by \autoref{thm:s:totally}, the geodesic equation on $\mathcal P\mathcal I^1_0$ is simply the restriction of the geodesic equation on $\ImmL_0$. By \autoref{thm:p:metric} one has for $c \in \mathcal P\mathcal I^1_0$ and $c_t\in \mathcal P^1_0$ that
\begin{align*}
G_c(c_t,c_t)&=\frac{1}{\ell_c} \sum_{j=1}^{n} \frac{1}{\ell^j}\langle c_t^{j+1}-c_t^{j}, c_t^{j+1}-c_t^{j}\rangle\,, \\
G_c(c,c_t)&=\frac{1}{\ell_c} \sum_{j=1}^{n} \frac{1}{\ell^j}\langle c_t^{j+1}-c_t^{j}, c^{j+1}-c^{j}\rangle\,.
\end{align*}
These are the first and second term of the geodesic equation \eqref{equ:m:geo}. The remaining term $\langle v,D_s c_{t}\rangle D_s c_{t}-\frac12\langle D_s c_{t},D_s c_{t}\rangle v$, to which $D_s^{-1}\circ\pi_0$ is applied, has the coordinate expression
\begin{equation*}
\frac{1}{(\ell^i)^3}\langle c^{i+1}-c^i,c_t^{i+1}-c_t^{i}\rangle (c_t^{i+1}-c_t^{i})
-\frac{1}{2(\ell^i)^3}\langle c_t^{i+1}-c_t^i,c_t^{i+1}-c_t^{i}\rangle (c^{i+1}-c^{i})\,.
\qedhere
\end{equation*}
\end{proof}

\begin{example}\label{ex:selfintersection}
Geodesics in $P\mathcal I^1_0$ may form self-intersections and may have non-constant winding number. An example, which is inspired by \cite[Figure~3]{younes2008metric}, is the following curve $c$ in $P\mathcal I^1_0$, which is depicted in \autoref{fig:selfintersection}:
\begin{equation*}
c^1=-c^3=\begin{pmatrix}\sin(t)\\0\end{pmatrix}, 
\qquad
c^2=-c^4=\begin{pmatrix}0\\ \cos(t)\end{pmatrix}.
\end{equation*}
This curve is a solution of the geodesic equation, as can be verified using \autoref{thm:g:geo}. It self-intersects at all multiples of $\pi/2$, and its winding number at all times $t$ without self-intersection equals $\operatorname{sgn}(\sin(2t))$.
\end{example}

\begin{figure}
\centering
{\footnotesize
\begin{psfrags}
\include{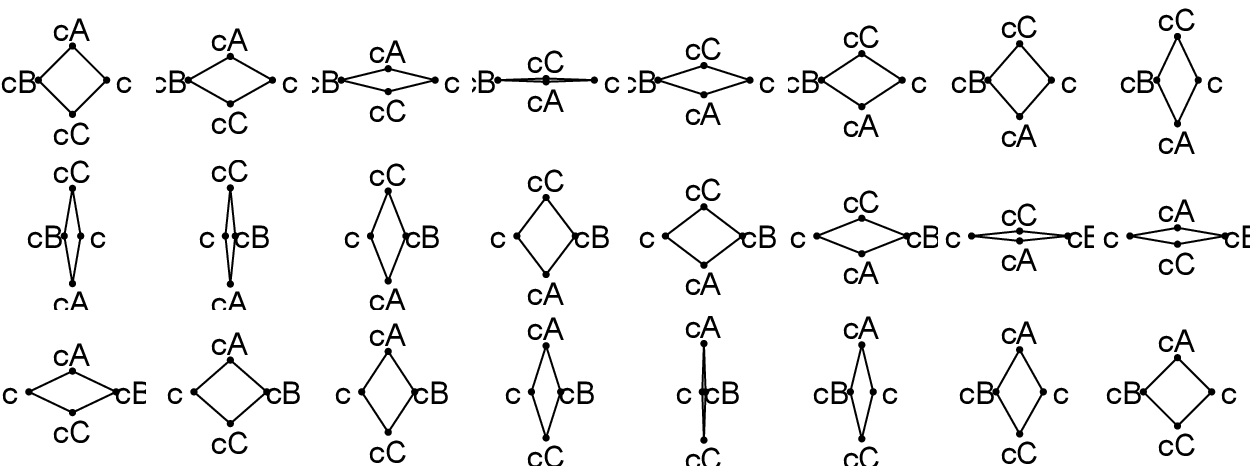}
\includegraphics[width=\textwidth]{selfintersection}
\end{psfrags}}
\caption{A closed geodesic in $P\mathcal I^1_0$ which has self-intersection and changes its winding number (c.f.\@ \autoref{ex:selfintersection}). See \href{https://arxiv.org/src/1702.04344v3/anc/selfintersection4.mp4}{here} or \href{https://www.mat.univie.ac.at/~michor/solitons-selfintersection4.mp4}{here} for an animation.}
\label{fig:selfintersection}
\end{figure}

\begin{remark}
The results of this section provide a powerful framework for solving the initial value problem for geodesics. One starts with an initial condition $(c,h) \in T\ImmL_0$. Assuming some additional smoothness, e.g.\@ $c,h \in W^{2,\infty}_0$, one can find for each $n \in \mathbb N$ a piecewise linear approximation $(c^{(n)},h^{(n)}) \in T\mathcal P^1_0$ built on a grid of $n$ points such that 
\begin{equation*}
\| c-c^{(n)} \|_{W^{1,\infty}_0} + \| h-h^{(n)} \|_{W^{1,\infty}_0} = O(1/n)\,.
\end{equation*}  
The geodesic equation with initial value $(c^{(n)},h^{(n)})$ is a second order ODE of dimension $(n-1)\times d$ and can be solved by standard methods with accuracy $1/n$ or better. As the exponential mapping is smooth, it follows that the $W^{1,\infty}$-distance between the true and discretized geodesics is  of order $1/n$. 

In dimension $d=2$ an alternative method is to solve the geodesic equation using the basic mapping of \cite{younes2008metric}; see \autoref{sec:basic} for how this would work in our setting. 
\end{remark}

\begin{figure}
\centering
{\footnotesize
\begin{psfrags}
\include{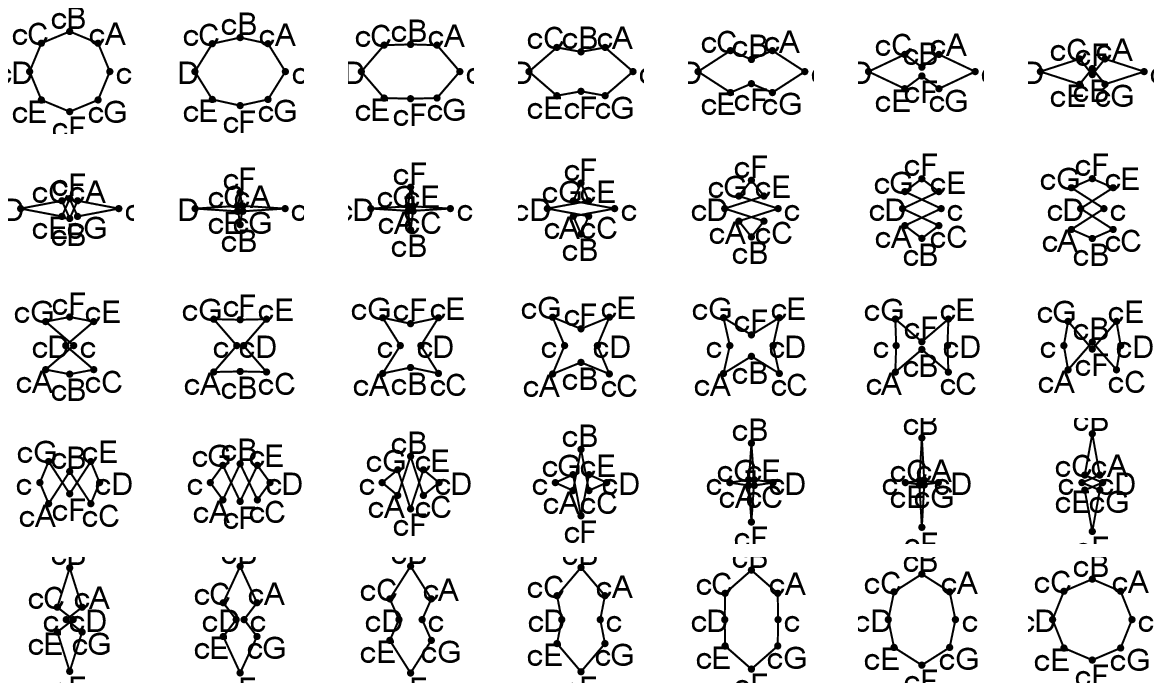}
\includegraphics[width=\textwidth]{selfintersection4}
\end{psfrags}}
\caption{A closed geodesic in $P\mathcal I^1_0$ which has self-intersection and changes its winding number (c.f.\@ \autoref{ex:selfintersection}). See \href{https://arxiv.org/src/1702.04344v3/anc/selfintersection8.mp4}{here} or \href{https://www.mat.univie.ac.at/~michor/solitons-selfintersection8.mp4}{here} for an animation.}
\label{fig:selfintersection4}
\end{figure}

\section{A Hamiltonian perspective}\label{sec:h}

The degeneracy of the bilinear form $G$ on $\mathcal P^1$ does not allow one to formulate the geodesic equation directly on this space, which is why we had to factor out translations in the first place. Interestingly, this problem does not occur in the Hamiltonian formulation. We will see below that Hamilton's equations make sense on all of $\mathcal P^1$, and that the solutions of Hamilton's equations project down to geodesics when translations are factored out. 

\begin{definition}\label{def:h:hamiltonian}
For the purpose of this section we view $G_c$, $c \in \mathcal P\mathcal I^1$, as a degenerate bilinear form $G_c:\mathcal P^1\times \mathcal P^1 \to \R$ and denote the corresponding linear operator by $\check G_c:\mathcal P^1\to(\mathcal P^1)^*$. Note that the relation $G_c(h,k)=G_c(\pi_1h,\pi_1k)$ can be expressed equivalently as
\begin{equation}\label{equ:h:Gcheck}
\check G_c = \pi_1^* \circ \check G_c|_{\mathcal P^1_0}\circ \pi_1: \mathcal P^1 \to \mathcal (P^1)^*\,.
\end{equation}
In analogy to this we define
\begin{equation}\label{equ:h:Kcheck}
\check K_c = \iota_1 \circ (\check G_c|_{\mathcal P^1_0})^{-1}\circ \iota_1^*: (\mathcal P^1)^* \to \mathcal P^1\,,
\end{equation}
where $\iota_1$ is given in \autoref{lem:p:Ds}, and we let $K_c:(\mathcal P^1)^*\times (\mathcal P^1)^*\to \R$ be the corresponding symmetric bilinear form. We call $K$ the extended cometric, and we define the Hamiltonian 
\begin{equation*}
H(c,\alpha) = \frac 12 K_c(\alpha,\alpha)\,.
\end{equation*}
\end{definition}

The meaning of $\check K$ is clarified by the following lemma.

\begin{lemma}\label{lem:h:pseudo}
$\check K_c$ is the Moore--Penrose pseudo-inverse of $\check G_c$ with respect to the $L^2(\ud \theta)$ scalar product on $\mathcal P^1$ and the dual scalar product on $(\mathcal P^1)^*$, i.e., 
\begin{align*}
\check G_c \check K_c\check G_c &= \check G_c\,, &
\check K_c \check G_c\check K_c &= \check K_c\, &
(\check K_c\check G_c)^\top &= \check K_c\check G_c\,, &
(\check G_c\check K_c)^\top &= \check G_c\check K_c\,.
\end{align*}
\end{lemma}

\begin{proof}
Formulas \eqref{equ:h:Gcheck} and \eqref{equ:h:Kcheck} and the identity $\pi_1 \iota_1 = \Id_{\mathcal P^1_0}$ imply that
\begin{equation}\label{equ:h:pseudo}\begin{aligned}
\check K\check G 
&= 
\iota_1 \circ (\check G_c|_{\mathcal P^1_0})^{-1}\circ \iota_1^* \circ \pi_1^* \circ \check G_c|_{\mathcal P^1_0}\circ \pi_1 
= \iota_1\pi_1\,,
\\
\check G\check K
&= 
\pi_1^* \circ \check G_c|_{\mathcal P^1_0}\circ \pi_1 \circ \iota_1 \circ (\check G_c|_{\mathcal P^1_0})^{-1}\circ \iota_1^*
= \pi_1^* \iota_1^*\,.
\end{aligned}\end{equation}
Similarly, one obtains in a further similar step that $\check G \check K\check G = \check G$ and $\check K \check G\check K = \check K$. This establishes the first two equations of the lemma. The remaining ones are satisfied because the mappings $\check K\check G=\iota_1\pi_1$ and $\check G\check K=\pi_1^*\iota_1^*$ are symmetric with respect to the $L^2(d\theta)$ scalar products on $\mathcal P^1$ and $(\mathcal P^1)^*$, respectively.
\end{proof}

We then have:
\begin{theorem}
Let $T>0$, and let $(c,\alpha):[0,T) \to \mathcal P^1\times (\mathcal P^1)^*$ be a solution of Hamilton's equations
\begin{equation*}
c_t = \p_\alpha H(c,\al)\,, \qquad
\al_t = -\p_c H(c,\al)\,.
\end{equation*}
Then $c$ is a critical point of the energy functional. If additionally $c(0) \in \mathcal P^1_0$, then $c(t) \in \mathcal P^1_0$ for all $t \in [0,T)$, and $c$ is a geodesic on the Riemannian space $(\mathcal P^1_0,G)$. 	Conversely, if $c:[0,T)\to \mathcal P^1_0$ is a geodesic and $\alpha = \check G_c c_t$, then $(c,\alpha)$ is a solution of Hamilton's equations. 
\end{theorem}
Note that the initial momentum $\al(0)$ can be arbitrary.

\begin{proof}
Letting $D_{(c,h)}$ denote the directional derivative at $c \in \mathcal P^1$ in the direction $h \in \mathcal P^1$, Hamilton's equations can be rewritten as
\begin{align*}
c_t &= \check K_c \al &
\al_t &= -\frac 12 (D_{(c,\cdot)} K)(\al, \al)\,.
\end{align*}
As the range of $\check K_c$ is $\mathcal P^1_0$, we see that $c(0) \in \mathcal P^1_0$ implies $c(t) \in \mathcal P^1_0$ for all $t$. Hamilton's first equation and \eqref{equ:h:pseudo} imply that for each $h \in \mathcal P^1_0$,
\begin{equation*}
G_c(c_t,h) = G_c(\check K_c \alpha,h) = \alpha(\check K_c\check G_c h) = \alpha(\iota_1\pi_1h) = \alpha(h)\,.
\end{equation*}
Together with the identity 
\begin{equation*}
D_{(c,h)}\check K=-\check K_c (D_{(c,h)}\check G)\check K_c,
\end{equation*}
which one obtains by applying $D_{(c,h)}$ to the identities
\begin{equation*}
\check K_c = \check K_c \pi_1^*\iota_1^* = \iota_1\pi_1\check K_c, 
\qquad
\check K_c = \check K_c \check G_c \check K_c,
\end{equation*}
and Hamilton's second equation this implies that
\begin{align*}
\big(G_c(c_t,h)\big)_t
&=
\alpha_t(h)
= 
-\frac 12 (D_{(c,h)} K)(\al, \al)
=
\frac12 (D_{(c,h)} G)(\check K_c \alpha, \check K_c \alpha)
\\&=
\frac12 (D_{(c,h)} G)(c_t, c_t)\,.
\end{align*}
Thus, for any smooth path $h:[0,1]\to \mathcal P^1_0$ satisfying  $h(0)=h(1)=0$, the derivative of the Riemannian energy (c.f.\@ \autoref{thm:m:geo}) vanishes,
\begin{align*}
\ud E(c).h 
&=
\frac12 \int_0^1 D_{(c,h)}\left(G_c(c_t,c_t)\right)\ud t
\\&=
\int_0^1 \left(\frac12 (D_{(c,h)} G)(c_t,c_t)+G_c(c_t,h_t)\right)\ud t
\\&=
\int_0^1 \left(\frac12 (D_{(c,h)} G)(c_t,c_t)-\big(G_c(c_t,\cdot)\big)_t\ h\right)\ud t=0\,,
\end{align*}
and $c$ is a geodesic with respect to the metric $G$ on $\mathcal P^1_0$. The converse statement follows by reversing the argument.
\end{proof}

\begin{lemma}\label{lem:h:extended_cometric}
An explicit formula for $K$, using the identifications of Definition~\ref{def:p:coordinates}, is given by 
\begin{equation*}
K_c(\alpha,\beta) = \sum_{i,j=1}^n K_{i,j}(c) \langle\alpha^i,\alpha^j\rangle_{\mathbb R^d},
\end{equation*}
where
\begin{align*}
K_{i,j}(c) 
&= 
\la^1 \la^{\max(i,j)} - \la^i\la^j
+ \frac{\ka^1}{n}(\la^i + \la^j) - \frac{\la^1}{n}(\ka^i + \ka^j) 
\\&\qquad
- \frac{(\ka^1)^2}{n^2} + \frac{\la^1}{n^2} \sum_{k=1}^n k^2 \ell^k \,,
\end{align*}
with
\begin{align*}
\la^i &= \sum_{j=i}^n \ell^j &
\ka^i &= \sum_{j=i}^n j \ell^j\,.
\end{align*}
\end{lemma}

\begin{proof}
We have
\begin{align*}
K_c(\al, \be) &= G_c\inv\left(\iota^\ast_1 \al, \iota^\ast_1 \be\right) \\
&= \sum_{i,j=1}^n \left( \la^1 \la^{\max(i,j)} - \la^i \la^j \right)
\left\langle (\iota^\ast_1 \al)^i, (\iota^\ast_1 \be)^j \right\rangle\,,
\end{align*}
together with
\[
\left(\iota^\ast_1 \al\right)^i = \al^i - \frac 1n \sum_{j=1}^n \al^j\,.
\]
The calculation proceeds via the following four identities.

{\bfseries Step 1.} $\displaystyle \sum_{i=1}^n \la^i  = \sum_{i=1}^n i \ell^i = \ka^1$.
This follows from
\begin{align*}
\sum_{i=1}^n \la^i = \sum_{i=1}^n \sum_{j=i}^n \ell^j = \sum_{j=1}^n \sum_{i=1}^j \ell^j = \sum_{j=1}^n j \ell^j = \ka^1\,.
\end{align*}

{\bfseries Step 2.} $\displaystyle \sum_{i=j+1}^n \la^i = \sum_{i=j+1}^{n} (i-j)\ell^i$.
This follows from
\begin{align*}
\sum_{i={j+1}}^n \la^i = \sum_{i={j+1}}^n \sum_{k=i}^n \ell^k 
= \sum_{k=j+1}^n \sum_{i=j+1}^k \ell^k = \sum_{k=j+1}^n (k-j)\ell^k\,.
\end{align*}

{\bfseries Step 3.} $\displaystyle \sum_{i,j=1}^n \left( \la^1 \la^{\max(i,j)} - \la^i \la^j \right) = \la^1 \sum_{k=1}^n k^2 \ell^k - (\ka^1)^2$. We start with
\begin{align*}
\sum_{i,j=1}^n \la^{\max(i,j)} 
&= \sum_{i=1}^n \left( i \la^i + \sum_{j=i+1}^n \la^j \right)
= \sum_{i=1}^n \sum_{j=i}^n i \ell^j + \sum_{i=1}^n \sum_{j=i+1}^n (j-i)\ell^j \\
&= \sum_{i=1}^n i \ell^i + \sum_{j=1}^n \sum_{i=1}^{j-1} j \ell^j = \sum_{i=1}^n i \ell^i
+ \sum_{j=1}^n j(j-1)\ell^j = \sum_{i=1}^n i^2 \ell^i\,.
\end{align*}
Therefore
\begin{align*}
\sum_{i,j=1}^n \left( \la^1 \la^{\max(i,j)} - \la^i \la^j \right)
&= \la^1 \sum_{i,j=1}^n \la^{\max(i,j)} - \left(\sum_{i=1}^n \la^i \right)^2 \\
&= \la^1 \sum_{k=1}^n k^2 \la^k - (\ka^1)^2\,.
\end{align*}

{\bfseries Step 4.} $\displaystyle \sum_{i=1}^n \la^1 \la^{\max(i,j)} - \la^i \la^j
= \la^1 \ka^j - \ka^1 \la^j$. We start with
\begin{align*}
\sum_{i=1}^n \la^{\max(i,j)} 
&= \sum_{i=1}^j \la^j + \sum_{i=j+1}^n \la^i
= j \la^j + \sum_{i=j+1}^n (i-j)\ell^i \\
&= \sum_{i=j+1}^n i \ell^i + j \la^j - j \la^{j+1}
= \sum_{i=j}^n i \ell^i = \ka^j\,.
\end{align*}
Therefore
\[
\sum_{i=1}^n \la^1 \la^{\max(i,j)} - \la^i \la^j
= \la^1 \ka^j - \ka^1 \la^j\,.
\]

To complete the proof it remains to combine the formulas for $G_c\inv$ and $\iota^\ast_1$ using the formulas derived in steps 3 and 4.
\end{proof}

\section{Relation to landmark spaces}\label{sec:l}

In this section we put the space of piecewise linear curves into the context of landmark spaces, which are important in shape analysis \cite{Bookstein1977,Kendall1984,Joshi2000}, and describe relations to the Large Deformation Diffeomorphic Metric Mapping (LDDMM) framework \cite{beg2005computing}, which is a widely used approach for defining metrics on landmark spaces. 

\begin{definition}
An ordered landmark is a tuple of pairwise distinct points $q^1, \dots, q^n$ in $\R^d$. The set of all landmarks is denoted by $\Land$; it is an open subset of $\R^{nd}$. Ordered landmarks can be seen as piecewise linear curves by connecting consecutive points via straight lines. Note that for landmarks all pairs of vertices are distinct, whereas for piecewise linear immersions only pairs of subsequent vertices are distinct. Thus, landmark space is an open subset of the space of piecewise linear immersions, i.e., $\Land\subset \mathcal P\mathcal I^1$, and the $H^1$-metric on $\Land$ is a well-defined non-negative (degenerate) bilinear form. The landmark space modulo translations is then given by
\begin{align*}
\Land_0 &= \left\{(q^1,\dots,q^n) \in \Land: \sum_{i=1}^n q^i = 0 \right\} \subset \mathcal P\mathcal I^1_0\,.
\end{align*}
\end{definition}

\begin{figure}
\centering
\begin{minipage}[c]{0.32\textwidth}
{\footnotesize
\begin{psfrags}
\include{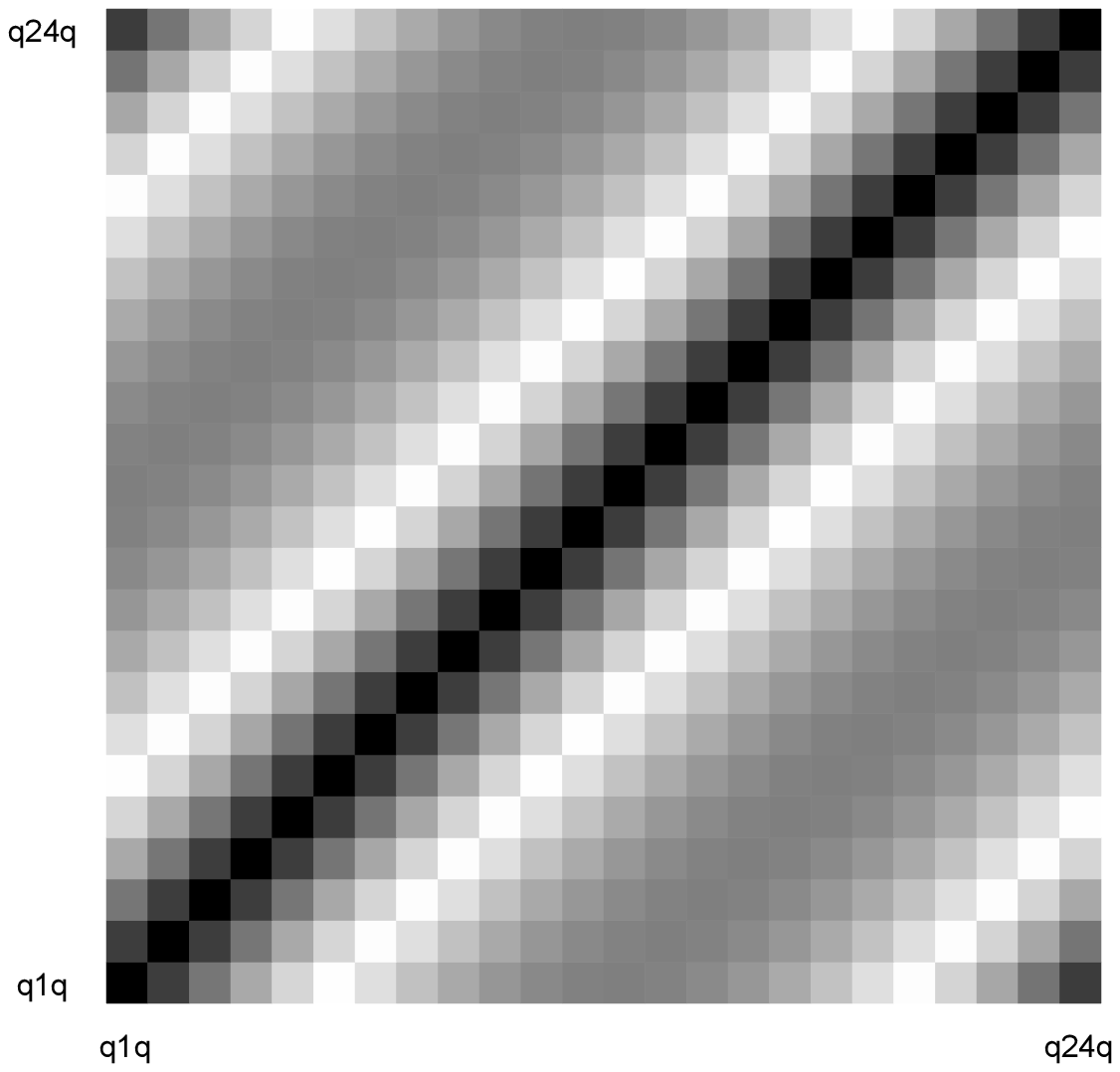}
\includegraphics[width=\textwidth]{kernel_h1_land}
\end{psfrags}}
\end{minipage}
\begin{minipage}[c]{0.32\textwidth}
{\footnotesize
\begin{psfrags}
\include{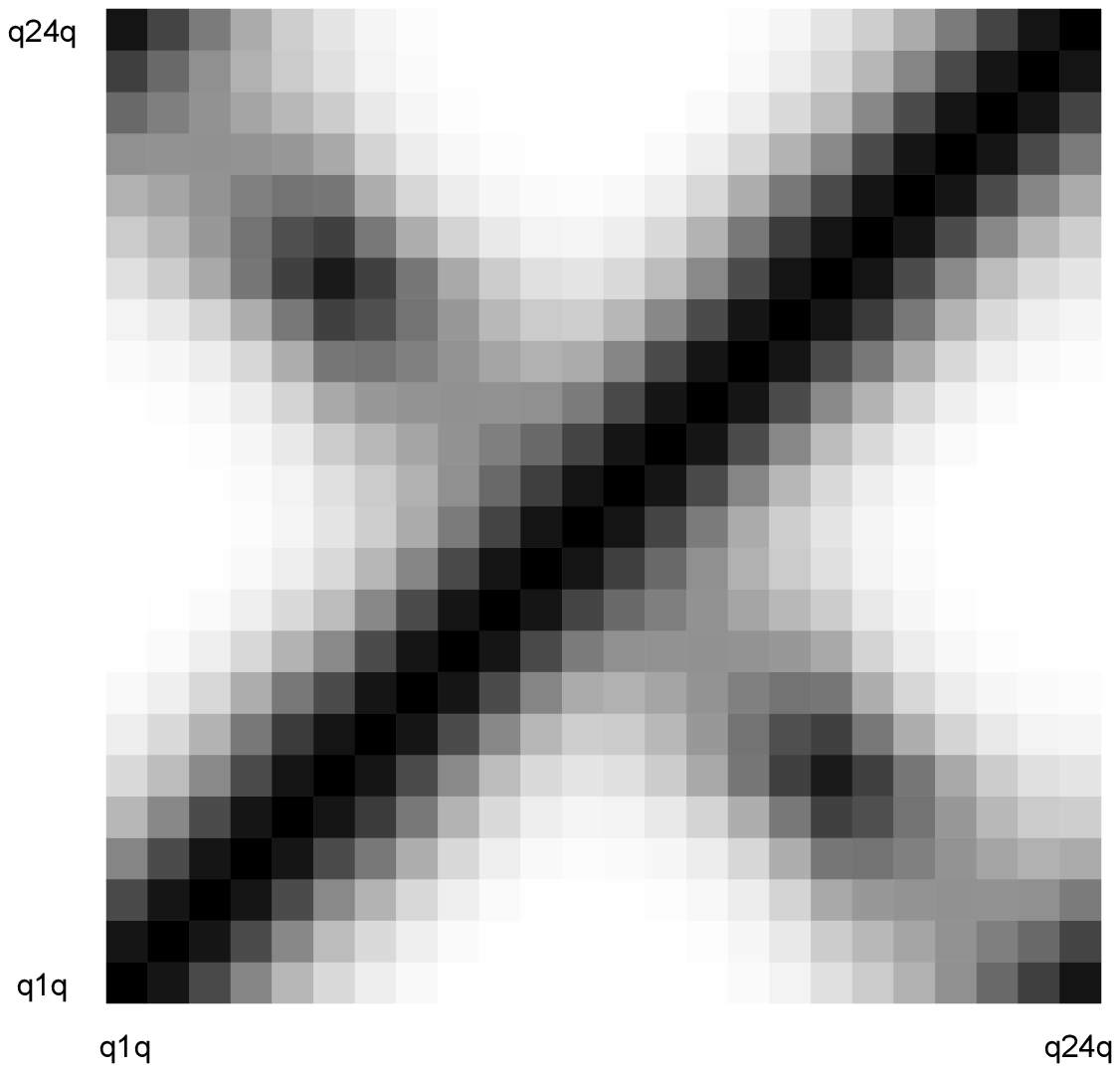}
\includegraphics[width=\textwidth]{kernel_gaussian_land}
\end{psfrags}}
\end{minipage}
\begin{minipage}[c]{0.32\textwidth}
{\footnotesize
\begin{psfrags}
\include{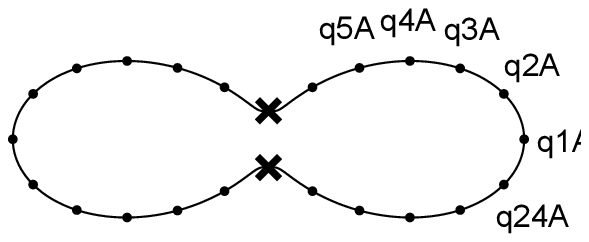}
\includegraphics[width=\textwidth]{kernel_curve}
\end{psfrags}}	
\end{minipage}
\caption[The kernel of the $H^1$ metric (left) compared to a Gaussian kernel (middle) at a specific landmark (right) on the space $\Land$. Note that the $H^1$ kernel does not ``see'' that the points marked by x have a small distance.]{The kernel of the $H^1$ metric (left) compared to a Gaussian kernel (middle) at a specific landmark (right) on the space $\Land$. Dark colors correspond to large values of the kernel. See \autoref{rem:l:interpretation} for an interpretation. 
}
\label{fig:kernels3}
\end{figure}

We now describe the construction of LDDMM metrics on landmark spaces. The approach is based on the paradigm of Grenander's pattern theory, where geometric objects are encoded via transformations acting on them. A metric on the transformation group then induces a metric on the space of geometric objects. In the LDDMM framework the transformation group is a group of diffeomorphisms equipped with a right invariant metric, which usually comes from a reproducing kernel Hilbert space. We refer to \cite{Younes2010} for further details.

\begin{definition}
Let $(H, \langle\cdot,\cdot\rangle_H)$ be a reproducing kernel Hilbert space of vector fields on $\R^d$ with kernel $\mathbf k^H:\R^d \x \R^d \to \R^{d\x d}$.
Provided that $H$ contains $C^\infty_c(\R^d,\R^d)$, the inner product $\langle\cdot,\cdot\rangle_H$ can be extended to a weak Riemannian metric on $\on{Diff}_c(\R^d)$ via right translation. This metric induces a unique metric $G^H$ and cometric $K^H$ on landmark space such that the action of the diffeomorphism group on a fixed template landmark is a Riemannian submersion. 
\end{definition}

The following lemma contrasts LDDMM and $H^1$-cometrics. 

\begin{lemma}
\label{lem:l:cometrics}
For each $q \in \Land$ the LDDMM cometric $K^H$ on $\Land$ is given by
\begin{equation*}
K^H_q=
\left( \begin{array}{ccc}
\mathbf k^H(q^1,q^1)  & \cdots & \mathbf k^H(q^1,q^n)\\
\vdots & \ddots  & \vdots\\
\mathbf k^H(q^n,q^1) & \cdots & \mathbf k^H(q^n,q^n)\end{array} \right) \in \R^{nd \x nd} \,,
\end{equation*}
and the extended $H^1$-cometric $K$ on $\Land$ (see \autoref{def:h:hamiltonian}) is given by
\begin{equation*}
K_q=
\left( \begin{array}{ccc}
K_{1,1}(q) I^{d\times d} & \cdots & K_{1,n}(q) I^{d\times d}\\
\vdots & \ddots  & \vdots\\
K_{n,1}(q)I^{d\times d} & \cdots & K_{n,n}(q)I^{d\times d}
\end{array} \right) \in \R^{nd \x nd} \,,
\end{equation*}
where $K_{i,j}(q) \in \R$ is given by \autoref{lem:h:extended_cometric} and $I^{d\times d}$ is the identity matrix of size $d$.
\end{lemma}

\begin{proof}
The formula for $K^H_q$ is due to \cite{Micheli2012}, and the one for $K_q$ can be seen from \autoref{lem:h:extended_cometric}.
\end{proof}

\begin{remark}\label{rem:l:interpretation}
The comparison of the cometrics in \autoref{lem:l:cometrics} reveals several differences. First, the $(i,j)$-th entry of the LDDMM cometric $K^H_q$ depends only on $q^i$ and $q^j$, whereas the $(i,j)$-th entry of the extended $H^1$-cometric $K_q$ depends on all of $q=(q^1,\dots,q^n)$. 

Second, the LDDMM cometric typically depends on the pairwise distances between all landmark points, whereas the $H^1$-cometric depends only on the distances between subsequent landmark points. This is illustrated in \autoref{fig:kernels3}, where the Gaussian LDDMM cometric with kernel $\mathbf k^H(q^i,q^j) = \exp(-|q^i-q^j|^2/2) I^{d\times d}$ is compared to the extended $H^1$-cometric. The left and middle plots show the scalar weights which appear in front of the matrices $I^{d\times d}$ in the expressions of the kernels $K_q$ and $K^H_q$ (cf.\@ \autoref{lem:l:cometrics}), and the right plot shows the landmark $q$. Note that there are off-diagonal dark regions in the plot of the LDDMM kernel, but not in the plot of the $H^1$-kernel. The reason is that in contrast to the LDDMM kernel, the $H^1$-kernel disregards that the landmark points marked by a cross in the right plot have a small distance. 
\end{remark}

\begin{figure}
\centering
{\footnotesize
\begin{psfrags}
\include{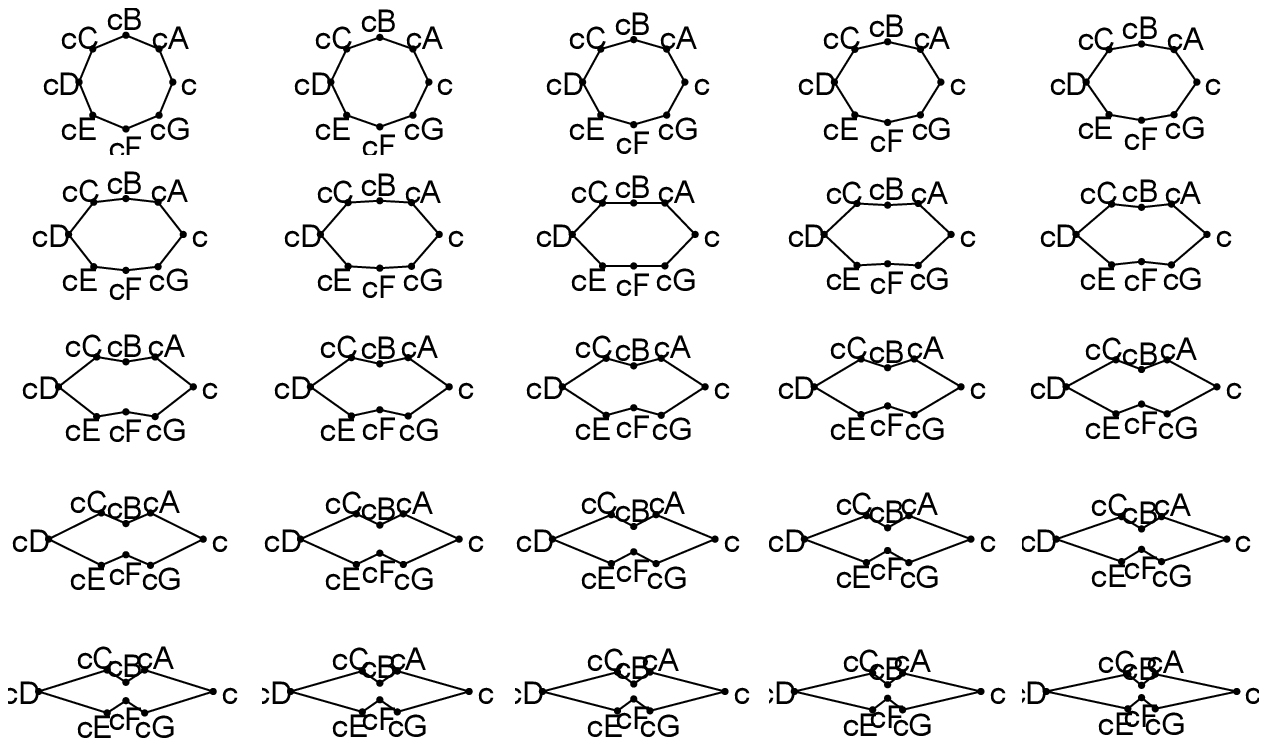}
\includegraphics[width=\textwidth]{lddmm1}
\end{psfrags}}
\caption{A geodesic with respect to the LDDMM metric with the same initial condition as in Figure~\ref{fig:selfintersection4}. Note that the LDDMM metric avoids landmark collisions; the landmarks never touch. See \href{https://arxiv.org/src/1702.04344v3/anc/lddmm8.mp4}{here} or \href{https://www.mat.univie.ac.at/~michor/solitons-lddmm8.mp4}{here} for an animation.}
\label{fig:lddmm1}
\end{figure}

\section{\texorpdfstring{Relation to the basic mapping of Younes et al. {\cite{younes2008metric}}}{Relation to the basic mapping of Younes et al.}}\label{sec:basic}

The basic mapping $\Phi$ of \cite{younes2008metric} is a locally isometric two-fold covering map from a certain Stiefel manifold to the manifold of closed unit-length smooth planar curves. In our setting, i.e., for parametrized closed Lipschitz curves, the basic mapping takes the following form:

\begin{lemma}\label{lem:basic}
Let $d=2$, $i=\sqrt{-1}$, and endow the manifold
\begin{multline*}
S := \bigg\{(e,f) \in L^\infty: \int_{S^1} e^2 - f^2 \ud\theta = \int_{S^1} ef  = 0,\,
\essinf_{\theta \in S^1} e^2+f^2>0\bigg\}\,,
\end{multline*}
with tangent bundle
\begin{multline*}
T_{(e,f)}S =
\bigg\{(\delta e,\delta f) \in L^\infty:
\int_{S^1} (e\delta  e-f\delta  f)\ud\theta =\int_{S^1} (e\delta f + f\delta e) \ud\theta = 0\bigg\}\,,
\end{multline*}
with the Riemannian metric
\begin{equation*}
G^S_{(e,f)}\big((\delta e,\delta f),(\delta e,\delta f)\big)
=
\int_{S^1} \big( \delta e(\theta)^2+\delta f(\theta)^2 \big) \ud \theta\,.
\end{equation*}
Under the identification $\mathbb R^2 \cong \mathbb C$, the mapping
\begin{equation*}
\Phi\colon S \to \ImmL_1, \qquad (e,f) \mapsto \frac12 \int_0^\cdot \big(e(\theta) + i f(\theta) \big)^2 \ud \theta,
\end{equation*}
is a smooth covering map and a local isometry.  
\end{lemma}

\begin{proof}
$S$ is a Banach submanifold of $L^\infty$ because the set
\begin{equation}\label{equ:basic:open}
\left\{(e,f) \in L^\infty: \essinf_{\theta \in S^1} |e(\theta)|^2+|f(\theta)|^2>0\right\}
\end{equation}
is open in $L^\infty$ (cf.~\autoref{thm:m:manifold}) and because the differential of the mapping
\begin{equation*}
L^\infty \to \mathbb R^2, \qquad (e,f) \mapsto \left(\int_{S^1} \big(e(\theta)^2-f(\theta)^2\big)\ud\theta , \int_{S^1}e(\theta)f(\theta)\right)
\end{equation*}
is surjective at any point in $S$, as can be seen by differentiation at the point $(e,f) \in S$ in the directions $(e,-f)$ and $(f,e)$. 

The mapping $\Phi$ is well-defined because the conditions
\begin{equation*}
\int_{S^1} \big(e(\theta)^2-f(\theta)^2\big)\ud\theta = \int_{S^1}e(\theta)f(\theta) = 0\,,
\end{equation*} 
readily imply that $\Phi(e,f)$ is a $2\pi$-periodic function. Moreover, $\Phi$ is smooth because it is a composition of bounded (multi-)linear mappings. 

To verify that $\Phi$ is a covering map, define for any $c \in \ImmL_1$ and $(e,f) \in S$,
\begin{multline*}
U(e,f) = \left\{r e^{i\phi} (e+if) \!:\! (r, \phi) \in L^\infty, \frac{1}{\sqrt 2} \!<\! \essinf_\theta r(\theta) \!\leq\! \esssup_\theta r(\theta) \!<\! \sqrt 2, \right.
\\
\left.-\frac{\pi}{2} < \essinf_\theta \phi(\theta)\leq \esssup_\theta \phi(\theta)<\frac{\pi}{2} \right\} \cap S,
\end{multline*}
\begin{multline*}
V(c) = \left\{\int_0^\cdot r e^{i\phi} c_\theta\ud\theta: (r, \phi) \in L^\infty, \frac12 < \essinf_\theta r(\theta)\leq \esssup_\theta r(\theta)<2, \right.
\\
\left.-\pi < \essinf_\theta \phi(\theta)\leq \esssup_\theta \phi(\theta)<\pi \right\} \cap \Lip_1 .
\end{multline*}
Then $U(e,f)$ is an open neighbourhood of $(e,f) \in S$, $V(c)$ is an open neighborhood of $c \in \ImmL_1$, and $\Phi$ maps $U(e,f)$ diffeomorphically to $V(\Phi(e,f))$. Moreover, for any distinct elements $(e,f)$ and $(\tilde e,\tilde f)$ of $\Phi^{-1}(c)$, the sets $U(e,f)$ and $U(\tilde e,\tilde f)$ are disjoint. To see this, note that there is a measurable function $\epsilon:S^1\to\{-1,1\}$ such that 
\begin{equation*}
(\tilde e(\theta),\tilde f(\theta))=\epsilon(\theta)(e(\theta),f(\theta))
\end{equation*}
holds for Lebesgue almost every $\theta \in S^1$. The set of all $\theta \in S^1$ with the property that $\epsilon(\theta)=-1$ has positive Lebesgue measure because $(e,f)\neq(\tilde e,\tilde f)$, and for any such $\theta$ the half planes
\begin{equation*}
\left\{re^{i\phi}\big(e(\theta)+if(\theta)\big):\frac{1}{\sqrt 2}<r<\sqrt 2, -\frac{\pi}{2}<\phi<\frac{\pi}{2}\right\}
\end{equation*}
and
\begin{equation*}
\left\{re^{i\phi}\big(\tilde e(\theta)+i\tilde f(\theta)\big):\frac{1}{\sqrt 2}<r<\sqrt 2, -\frac{\pi}{2}<\phi<\frac{\pi}{2}\right\}
\end{equation*}
don't intersect. It follows that $U(e,f)\cap U(\tilde e,\tilde f)=\emptyset$. Thus, for any $c \in \ImmL_1$ the set $\Phi^{-1}(V(c))$ is a disjoint union of open sets, which are diffeomorphic to $V(c)$, and we have shown that $\Phi$ is a covering map. 

To see that $\Phi$ is a local isometry, note that the derivative of $\Phi$ is given by
\begin{align*}
T_{(e,f)}\Phi(\delta e,\delta f) 
&=
\int_0^\cdot (e+if)(\delta e+i\delta f)\ud \theta\,.  
\end{align*}
Therefore, 
\begin{align*}
D_s T_{(e,f)}\Phi(\delta e,\delta f) 
&=
\frac{(e+if)(\delta e+i\delta f)}{|e+if|^2}\,.
\end{align*}
This implies that $\Phi$ is a local isometry:
\begin{align*}
G_{\Phi(e,f)}\big(T_{(e,f)}\Phi(\delta e,\delta f), &T_{(e,f)}\Phi(\delta e,\delta f)\big)
\\ &=
\int_{S^1} \left|\frac{(e+if)(\delta e+i\delta f)}{|e+if|^2}\right|^2 |e+if|^2 \!\ud\theta
\\&=
\int_{S^1} |e+if|^2\ud\theta 
= 
G^S_{(e,f)}\big((\delta e,\delta f),(\delta e,\delta f)\big)\,.
\qedhere
\end{align*}
\end{proof}

\begin{corollary}
In the setting of \autoref{lem:basic} the following statements hold:
\begin{enumerate}[label=(\roman*)]
\item Under the mapping $\Phi$, geodesics on $S$ project down to geodesics on $\ImmL_1$, and conversely, geodesics on $\ImmL_1$ can be lifted (uniquely up to the choice of a measurable function $S^1\to\{-1,1\}$) to geodesics on $S$.

\item Let $\operatorname{St}$ denote the Stiefel manifold of $L^2(\ud \theta)$-orthonormal pairs $(e,f) \in S$. Then $\Phi$ restricts to a smooth covering map and local isometry
\begin{equation*}
\Phi\colon \operatorname{St} \to \left\{c \in \ImmL_1: \ell_c = 1\right\}.
\end{equation*}
\end{enumerate}
\end{corollary}

\begin{proof}
This follows trivially from \autoref{lem:basic}; cf.\@ \cite{younes2008metric}.
\end{proof}

\begin{remark} 
The space of unit length curves can be considered either as a submanifold of $\ImmL_1$ or as a quotient of $\ImmL_1$ modulo scalings. The submanifold and quotient metrics coincide because the scaling momentum $\partial_t (\operatorname{log} \ell(t))$ of the action of the scaling group is invariant. 
Therefore, geodesics with respect to the submanifold metric, which are studied in \cite{younes2008metric}, are  geodesics in the space of immersions  modulo scalings. The submanifold of unit length curves is, however, not totally geodesic in $\ImmL_1$.
Therefore, geodesics with respect to the submanifold metric are not geodesics in $\ImmL_1$. 
\end{remark}

\appendix

\section{Smoothness of the arc length derivative}\label{sec:a}

The aim of this section is to show that the mappings $c\mapsto|c_\theta|$ and $c\mapsto|c_\theta|^{-1}$ are smooth, where the subscript denotes the derivative with respect to $\theta\in S^1$. This is used in \autoref{sec:m} to showed that the first order Sobolev metric is smooth on the space $\ImmL_0(S^1,\R^d)$. We present two proofs: one using convenient calculus and the other one directly using Fr\'echet derivatives on Banach spaces. The strategy of the first proof is presented here for the first time and is of independent interest. 

\subsection*{Proof using convenient calculus}

\begin{result}{\rm\cite[\mbox{4.1.19}]{FK88}}
\label{res:a:smooth}
Let $c:\mathbb R\to E$ be a curve in a convenient vector space $E$. Let 
$\mathcal{V}\subset E'$ be a subset of bounded linear functionals such that 
the bornology of $E$ has a basis of $\sigma(E,\mathcal{V})$-closed sets. 
Then $c$ is smooth if and only if the following property holds:
\begin{itemize}
\item There exist locally bounded curves $c^{k}:\mathbb R\to E$ such that $\ell\circ c$ is smooth $\mathbb R\to \mathbb R$ with $(\ell\circ c)^{(k)}=\ell\circ c^{k}$, for each $\ell\in\mathcal V$. 
\end{itemize}
Moreover, if $E$ is reflexive, then for any point separating subset
$\mathcal{V}\subset E'$ the bornology of $E$ has a basis of 
$\si(E,\mathcal{V})$-closed subsets, by {\rm \cite[\mbox{4.1.23}]{FK88}}.
\end{result}

For any path $c$ in some space of $\R^d$-valued functions on $S^1$, we write $\hat c$ for the corresponding mapping $\hat c:\R \times S^1\to\R^d$.

\begin{lemma}
\label{lem:a:smooth1}
The space $C^\infty(\mathbb R,\Lip)$ consists of all mappings $\hat c:\mathbb R\x S^1\to \mathbb R^d$ with the following property:
\begin{itemize}
\item For fixed $\th\in S^1$ the function $x\mapsto \hat c(x,\th)\in \mathbb R^d$ is smooth and each derivative $x\mapsto \p_x^k \hat c(x,\;)$ is a locally bounded curve $\mathbb R\to \Lip$.
\end{itemize}
\end{lemma}

\begin{proof}
The space $\Lip$ is linearly isomorphic to the space $L^\infty$ via the isomorphism
\begin{equation}\label{equ:c:isomorphism}\begin{gathered}
L^\infty\to \Lip
\\
f \mapsto 
\Big (\th\mapsto\int_{0}^\th \Big(f(\al)-\frac1{2\pi}\int_{S^1}f(\be)d\be\Big)\ud\al + \frac1{2\pi}\int_{S^1}f(\be)\ud\be\Big)
\end{gathered}\end{equation}
Thus, $\Lip$ is isomorphic to the dual space of $L^1$. We take $\mathcal V$ as the set of directional point evaluations $\on{ev}^\la_\th := \langle\lambda,\delta_\theta(\cdot)\rangle_{\mathbb R^d}$ for $\th\in S^1$ and $\la \in \R^d$. Then $\mathcal V$ can be seen as a subset of $L^1$ using the isomorphism \eqref{equ:c:isomorphism}. Therefore, the topology $\si(\Lip,\mathcal V)$ is coarser on the unit ball $\bigcirc \Lip$ than the weak$^*$-star topology, for which $\bigcirc \Lip$ is compact. As $\si(\Lip,\mathcal V)$ is Hausdorff, the unit ball $\bigcirc \Lip$ is compact for $\si(\Lip,\mathcal V)$, thus $\si(\Lip,\mathcal V)$-closed.
So the condition of \autoref{res:a:smooth} is satisfied, and the statement of the lemma follows.
\end{proof}

\begin{lemma}
\label{lem:a:smooth2}
The space $C^\infty(\mathbb R,L^\infty)$ consists of all  sequences of locally bounded mappings $\hat c_k: \mathbb R\x S^1\to \mathbb R^d$ such that:
\begin{itemize}
\item For fixed $f\in C^\infty(S^1)$  each function $x\mapsto \int_{S^1}f(\th)\hat c_k(x,\th)\,d\th\in \mathbb R^d$ is smooth and  $\p_x \int_{S^1}f(\th)\hat c_k(x,\th)\,d\th = \int_{S^1}f(\th)\hat c_{k+1}(x,\th)\,d\th$.
\end{itemize}
\end{lemma}

\begin{proof}
The topology $\si(L^\infty, C^\infty)$ is coarser than 
$\si(L^\infty, L^1)$ for which the unit ball 
$\bigcirc L^\infty$ is compact. Since  $\si(L^\infty, C^\infty)$ is Hausdorff, the unit ball $\bigcirc L^\infty$ is also compact for the topology $\si(L^\infty, C^\infty)$ and thus $\si(L^\infty, C^\infty)$-closed.
So the condition of \autoref{res:a:smooth} is satisfied, and the statement of the lemma follows.
\end{proof}

\begin{corollary}\label{cor:a:smooth}
The mappings $c\mapsto |c_\theta|$ and $c\mapsto |c_\theta|^{-1}$ are smooth from $\ImmL$ to $L^\infty(S^1)$.
\end{corollary}

\begin{proof}
We have to check that $c\mapsto |c_\theta|$ and $c\mapsto |c_\theta|^{-1}$ map smooth curves to smooth curves. So let $c:\mathbb R\to \ImmL(S^1,\mathbb R^d)$ be a smooth curve. By \autoref{lem:a:smooth1} $x\mapsto \hat c(x,\theta)$ is smooth for each $\theta\in S^1$, and each derivative $x\mapsto \partial^k_x \hat c(x,\cdot)$ is a locally bounded curve in $W^{1,\infty}(S^1,\R^d)$. Then $\esssup_\th |\hat c_\th(x,\th)|$ and $\esssup_\th \frac1{|\hat c_\th(x,\th)|}$ are bounded locally uniformly in $x\in \mathbb R$. It follows that for each $f\in C^\infty(S^1)$ and $k \ge 0$,
\begin{align*}
\p_x\int_{S^1}f(\th) \p_x^k |\hat c_\th(x,\th)|\ud\th &= \int_{S^1}f(\th)\p_x^{k+1}|\hat c_\th(x,\th)|\ud\th \,,
\\
\p_x\int_{S^1}f(\th) \p_x^k \Big(\frac1{|\hat c_\th(x,\th)|}\Big)\ud\th &= \int_{S^1}f(\th)\p_x^{k+1}\Big(\frac1{|\hat c_\th(x,\th)|}\Big)\ud\th \,.
\end{align*}
Thus, we have verified the conditions of \autoref{lem:a:smooth2}, and $|c_\theta|$ and $|c_\theta|\inv$ are smooth curves in $L^\infty(S^1)$. 
\end{proof}

\subsection*{Proof using Fr\'echet derivatives}

The following is an Omega lemma on the space of essentially bounded functions. 

\begin{lemma}\label{lem:a:frechet}
Let $(\Omega,\mathcal F,\mu)$ be a measure space, let $E$ and $F$ be Euclidean vector spaces, let $U$ be an open subset of $E$, and let $f \in C^\infty(U,F)$ be a smooth function. Then 
\begin{equation*}
f_*\colon \left\{h \in L^\infty(\Omega,E); \essinf_{\omega \in \Omega}\inf_{x \in U^c}\|h(\omega)-x\|_E > 0\right\} \to L^\infty(\Omega,F),\; h \mapsto f \circ h,
\end{equation*} 
is a smooth mapping defined on an open subset of $L^\infty(\Omega,E)$.
\end{lemma}

\begin{proof}
Let $h$ be in the domain of $f^*$. Then the open ball
\begin{equation*}
B(h):=\left\{k \in L^\infty(\Omega,E); \|h-k\|_{L^\infty(\Omega,E)} < \essinf_{\omega\in\Omega}\inf_{x\in U^c} \|h(\omega)-x\|_E\right\} 
\end{equation*}  
also belongs to the domain of $f_*$, and the essential range of $h$ is contained in the compact set
\begin{equation*}
\left\{x \in E; \|x\|_E\leq \|h\|_{L^\infty(\Omega,E)}, \inf_{y \in U^c} \|x-y\|_E \geq \essinf_{\omega \in \Omega}\inf_{y\in U^c} \|h(\omega)-y\|_E \right\}. 
\end{equation*}
As this holds for all $h$, the domain of $f^*$ is open, and the range of $f_*$ is contained in $L^\infty(\Omega,F)$.

We will prove by induction on $n \in \mathbb N$ that $f_*$ is $n$ times Fr\'echet differentiable with Fr\'echet derivative
\begin{equation*}
f_*^{(n)}(h_0)(h_1,\dots,h_n) 
= 
f^{(n)}(h_0)(h_1,\dots, h_n).
\end{equation*}
Note that this is well-defined because $f^{(n)}(h_0)$ belongs to $L^\infty(\Omega,(E^*)^{\otimes n}\otimes F)$ by what we have just shown and because multiplication of $L^\infty$ functions is continuous. For $n=0$ there is nothing to prove. Assume the inductive hypothesis that the statement holds for $n$, let $h_0$ belong to the domain of $f_*$, let $\tilde h_0 \in B(h_0)$, let $K$ be the compact set given by
\begin{multline*}
\bigg\{x \in E; \|x\|_E\leq \|h_0\|_{L^\infty(\Omega,E)}+\|\tilde h_0\|_{L^\infty(\Omega,E)}, 
\\
\inf_{y \in U^c} \|x-y\|_E \geq \essinf_{\omega \in \Omega}\inf_{y\in U^c} \|h(\omega)-y\|_E - \|\tilde h_0-h_0\|_{L^\infty(\Omega,E)}\bigg\},
\end{multline*}
and let $h_1,\dots,h_n \in L^\infty(\Omega,E)$. Then it holds that 
\begin{align*}
\hspace{2em}&\hspace{-2em}
\bigg\|f_*^{(n)}(\tilde h_0)(h_1,\dots,h_n)-f_*^{(n)}(h_0)(h_1,\dots,h_n)
\\&\qquad\qquad
-f^{(n+1)}(h_0)(\tilde h_0-h_0,h_1,\dots,h_n)\bigg\|_{L^\infty(\Omega,F)}
\\&=
\bigg\|\int_0^1 f^{(n+1)}\big((1-t)h_0+t\tilde h_0)(\tilde h_0-h_0,h_1,\dots,h_n)\ud t
\\&\qquad\qquad
-f^{(n+1)}(h_0)(\tilde h_0-h_0,h_1,\dots,h_n) \bigg\|_{L^\infty(\Omega,F)}
\\&\leq
\sup_{x,y \in K} \left\|f^{(n+1)}(x)-f^{(n+1)}(y)\right\|_{(E^*)^{\otimes(n+1)}\otimes F} \|\tilde h_0-h_0\|_{L^\infty(\Omega,E)}
\\&\qquad\qquad
\cdot\|h_1\|_{L^\infty(\Omega,E)}\dots \|h_n\|_{L^\infty(\Omega,E)}.
\end{align*}
This proves the statement for $n+1$. Thus, we have shown by induction that $f_*$ is infinitely Fr\'echet differentiable. 
\end{proof}

\begin{corollary}\label{cor:a:frechet}
The mappings $c\mapsto |c_\theta|$ and $c\mapsto |c_\theta|^{-1}$ are smooth from $\ImmL$ to $L^\infty(S^1)$.
\end{corollary}

\begin{proof}
This follows from \autoref{lem:a:frechet} applied to the Lebesgue space $\Omega=S^1$, $E=\mathbb R^d$, $U = \mathbb R^d \backslash \{0\}$, $F=\mathbb R$, and $f(x)=|x|$ or $f(x)=|x|^{-1}$, respectively, using that $\partial_\theta\colon \ImmL \to L^\infty$ is a bounded linear map.
\end{proof}

\section{The cometric on the space of Lipschitz immersions}
\label{sec:b}

In this part we want to describe the momentum associated to a velocity $h\in T_c\ImmL$ and use this to calculate the cometric on the space of Lipschitz immersions. 

\begin{lemma}\label{lem:b:mom}
For any $c \in \ImmL_0$ and $h \in \Lip_0$, the momentum $\check G_c(h) \in (\Lip_0)^*$ is the $\mathbb R^d$-valued distribution on $S^1$ given by
\begin{equation*}
\check G_c(h) = \frac{1}{\ell_c}D_s^*(D_s h\ud s)=-\frac{1}{\ell_c} D_s^2 (h \ud s)\,.
\end{equation*}
\end{lemma}

\begin{proof}
The first expression is clear from the definition: as $D_s^*$ applied to a distribution $\alpha$ is defined as $\alpha\circ D_s$, one has
\begin{align*}
\left\langle\frac{1}{\ell_c}D_s^*(D_s h\ud s),k\right\rangle_{(\Lip_0)^*,\Lip_0}
&=
\frac{1}{\ell_c}\left\langle D_s h\ud s,D_s k\right\rangle_{(W^{1,\infty}_0)^*,W^{1,\infty}_0}
\\&=
\frac{1}{\ell_c}\int_{S^1} \left\langle D_s h ,D_s k\right\rangle_{\R^d} \ud s
=
G_c(h,k).
\end{align*}
The second relation is obtained by integration by parts: for each smooth $k$,
\begin{equation*}
\check G_c(h)(k)=G_c(h,k)=\frac{1}{\ell_c}\int \langle D_s h,D_s k\rangle \ud s =-\frac{1}{\ell_c}\int \langle  h,D_s^2 k\rangle \ud s 
\,.
\end{equation*}
Note that $-D_s$ is the adjoint of $D_s$ with respect to $L^2(\ud s;\R^d)$. Thus, the statement of the lemma follows from the definition of distributional derivatives.
\end{proof}

We will now describe the cometric. Recall that $(L^\infty)^*$ is isomorphic to the Banach space $\on{ba}$ of $\R^d$-valued finitely additive set functions on the Lebesgue $\sigma$-algebra of $S^1$ which vanish on Lebesgue null sets, endowed with the total variation norm \cite[Theorem~IV.8.16]{dunford1958linear}. If an element of $\on{ba}$ is countably additive, then the Radon-Nikodym derivative with respect to $\ud s$ is well-defined. Moreover, recall that the smooth cotangent space is defined as $\check G_c(T_c\ImmL_0) \subseteq T^*_c\ImmL_0$.

\begin{lemma}\label{lem:b:cometric}
A covector $\alpha \in T^*_c\ImmL_0$ belongs to the smooth cotangent space if and only if the following two properties hold:
\begin{enumerate}
\item The set function $(D_s^{-1})^*\alpha \in \on{ba}$ is countably additive, i.e., it is an absolutely continuous vector measure, and
\item The Radon-Nikodym derivative of $(D_s^{-1})^*\alpha$ with respect to the measure $\!\ud s$ is in $L^\infty$.
\end{enumerate}
If $\alpha$ and $\beta$ are in the smooth cotangent space at $c$, then 
\begin{equation*}
G_c^{-1}(\alpha,\beta)=\ell_c \int_{S^1}\left\langle\frac{(D_s^*)^{-1}\alpha}{\!\ud s},\frac{(D_s^*)^{-1}\beta}{\!\ud s}\right\rangle\ud s\,.
\end{equation*}
\end{lemma}

\begin{proof}
If $\alpha$ is a smooth covector, then $\alpha = \ell_c^{-1}D_s^*(D_s h \ud s)$ for some $h \in \Lip_0$. Therefore, $(D_s^{-1})^*\alpha=\ell_c^{-1}D_s h \ud s$ is a countably additive set function, and $(D_s^{-1})^*\alpha/\ud s =\ell_c^{-1}D_s h \in L^\infty$. Conversely, assume that $(D_s^{-1})^*\alpha$ is countably additive and $k:=(D_s^{-1})^*\alpha/\ud s \in L^\infty$. Then $\int_{S^1} k\ud s=0$ by the definition of $\on{range} (D_s^{-1})^*=(W^{0,\infty}_0)^*$, which means that $k=D_sh$ for some $h \in \Lip_0$. Thus, $\alpha=D_s^*(D_sk \ud s)=\ell_c\check G_c(h)$ is in the smooth cotangent space. This shows the first statement. To show the formula for $G^{-1}$, let $\alpha = \check G_c(h)$ and $\beta=\check G_c(k)$. Then 
\begin{align*}
G_c^{-1}(\alpha,\beta)
&=
G_c(h,k)
=
\frac{1}{\ell_c}\int_{S^1}\langle D_sh,D_sk\rangle \ud s 
\\&= 
\ell_c\int_{S^1}\left\langle\frac{(D_s^*)^{-1}\alpha}{\!\ud s},\frac{(D_s^*)^{-1}\beta}{\!\ud s}\right\rangle\ud s\,.
\qedhere
\end{align*}  
\end{proof}

\printbibliography

\end{document}